\documentclass{conm-p-l}

\usepackage{latexsym, enumitem, stackrel}
\usepackage{rotating}
\usepackage{amsfonts}
\usepackage{amssymb}
\usepackage{amsmath, amsthm}
\usepackage{graphics, color}
\definecolor{darkblue}{rgb}{0,0,0.4}
\usepackage[colorlinks=true, citecolor=darkblue, filecolor=darkblue,
linkcolor=darkblue, urlcolor=darkblue]{hyperref}
\input xy
\xyoption{all}
\DeclareMathOperator*{\colim}{colim}

\DeclareMathOperator*{\hocolim}{hocolim}
\DeclareMathOperator*{\hocofib}{hocofib}

\newtheorem{theorem}{Theorem}[section] 
\newtheorem{lemma}[theorem]{Lemma}     
\newtheorem{corollary}[theorem]{Corollary}
\newtheorem{proposition}[theorem]{Proposition}

\theoremstyle{definition}
\newtheorem{definition}[theorem]{Definition}
\newtheorem{remark}[theorem]{Remark}
\newtheorem{question}[theorem]{Question}
\newtheorem{example}[theorem]{Example}

\newtheorem{construction}[theorem]{Construction}
\newcommand{\Irr}{\mathrm{Irr}^+}
\newcommand{\Sum}{\mathrm{Sum}}
\newcommand{\PAct}{{\bf PAct}}
\newcommand{\BPAct}{{\bf BPAct}}

\newcommand{\Top}{{\bf CGTop}}

\newcommand{\PCat}{{\bf PCat}}

\newcommand{\br}[1]{\left[ #1 \right]}
\newcommand{\Hol}{\mathrm{Hol}}
\renewcommand{\br}[1]{\left[ #1 \right]}
\newcommand{\bt}{\bullet}

\newcommand{\rh}{\widetilde{\Pi}}

\newcommand{\bbC}{\mathbb{C}}

\newcommand{\bbN}{\mathbb{N}}

\newcommand{\bbR}{\mathbb{R}}

\newcommand{\bbZ}{\mathbb{Z}}

\newcommand{\bbK}{\mathbb{K}}
\newcommand{\bbQ}{\mathbb{Q}}
\newcommand{\ignore}[1]{}
\newcommand{\A}{\mathcal{A}}
\newcommand{\B}{\mathcal{B}}
\newcommand{\mcC}{\mathcal{C}}

\newcommand{\mcG}{\mathcal{G}}

\newcommand{\mcK}{\mathcal{K}}

\newcommand{\M}{\mathcal{M}}

\newcommand{\qcd}{\bbQ cd}

\newcommand{\T}{\mathcal{T}}

\newcommand{\V}{\mathcal{V}}

\newcommand{\bS}{\mathbf{S}}

\newcommand{\leqs}{\leqslant}
\newcommand{\geqs}{\geqslant}
\newcommand{\heq}{\simeq}

\newcommand{\maps}{\longrightarrow}

\newcommand{\injects}{\hookrightarrow}

\newcommand{\isom}{\cong}
\newcommand{\cross}{\times}

\newcommand{\wt}[1]{\widetilde{#1}} 

\newcommand{\Rdef}{R^{\mathrm{def}}}

\newcommand{\Rep}{\mathrm{Rep}}

\newcommand{\Hom}{\mathrm{Hom}}

\newcommand{\GL}{GL}
 
\newcommand{\U}{U}
\newcommand{\hU}{hU}

\newcommand{\K}{K^{\mathrm{def}}}

\newcommand{\hofib}{\mathrm{hofib}}

\newcommand{\Map}{\mathrm{Map}}

\newcommand{\flatc}{\mathcal{A}^{\flat}}

\newcommand{\Gr}{\mathrm{Gr}}

\newcommand{\ku}{\mathbf{ku}}
\newcommand{\Susp}{\Sigma}

\newcommand{\xmaps}{\xrightarrow}
\newcommand{\srm}[1]{\stackrel{#1}{\maps}}
\newcommand{\srt}[1]{\stackrel{#1}{\to}}
\newcommand{\sm}{\wedge}

\newcommand{\goesto}{\mapsto}

\newcommand{\rK}{\wt{K}^\mathrm{def}}

\newcommand{\ol}[1]{\overline{#1}}
\def\co{\colon\thinspace}

\newcommand{\Img}{\mathrm{Im}}
 
\newcommand{\e}{\emph}

 \newcommand{\ra}{\rangle}
\newcommand{\la}{\langle}
 
\newcommand{\defn}{\mathrel{\mathop :}=}

\title[The topological Atiyah--Segal map]
 {The topological Atiyah--Segal map}

\author{Daniel A. Ramras}
\thanks{The author was partially supported by the Simons Foundation (\#279007 and \#579789).}
 
\address{IUPUI, 402 N Blackford St, Indianapolis, IN, 46032}
\email{dramras@iu.edu}

\begin{document}

\begin{abstract}  Associated to each finite dimensional linear representation of a group $G$, there is a vector bundle over the classifying space $BG$.  This construction was studied extensively for compact groups by Atiyah and Segal. We introduce a homotopy theoretical framework for studying the Atiyah--Segal construction in the context of infinite discrete groups, taking into account the topology of representation spaces.   
 
We explain how this framework relates to the Novikov conjecture, and we consider applications to spaces of flat connections on the over the 3--dimensional Heisenberg manifold and families of flat bundles over classifying spaces of groups satisfying Kazhdan's property (T).
\end{abstract}

\maketitle

\tableofcontents

\section{Introduction}

In  the 1950s and 1960s, Atiyah and Segal studied a construction that associates a vector bundle to each (complex) representation of a group $G$.  This construction yields a map
$$R[G] \maps K^* (BG),$$
from the complex representation ring of $G$ to the complex $K$--theory of its classifying space $BG$.  
The classical Atiyah--Segal Completion Theorem~\cite{Atiyah-Segal} states that when $G$ is a compact Lie group, this map becomes an isomorphism after completing $R[G]$ at its augmentation ideal.

We introduce an analogous map for infinite discrete groups, where the topology of the representation spaces $\Hom(G, U(n))$ plays a key role.  Loosely speaking, this map assigns to a spherical family of representations 
$$\rho \co S^m \to \Hom(G, U(n))$$ the $K$--theory class of the associated vector bundle 
$$E_{\rho} \to S^m \cross BG$$
 with holonomy $\rho$.  See Section~\ref{TAS-sec} for the definition of $E_{\rho} $, and a more precise statement along these lines (Theorem~\ref{TAS}).  The construction ${\rho} \goesto E_{\rho}$ was considered previously in~\cite{RWY}, where it was used to study the (strong) Novikov conjecture, and in~\cite{Baird-Ramras-smoothing}, where it was used to obtain cohomological lower bounds on the 
homotopy groups of spaces of flat, unitary connections.

The appropriate context for this construction is that of \e{deformation} $K$--\e{theory}.  
The most down-to-earth description of
the deformation $K$--theory of a group $G$ 
is that it arises from the homotopy group completion $\Omega B\Rep(G)$, where $\Rep(G)$ is the topological monoid
$$\Rep(G) = \coprod_n \Hom(G, \U(n)),$$
with block sum of matrices as the monoid operation. 
In particular, in positive degrees the reduced deformation $K$--theory groups $\rK_* (G)$ agree with the homotopy groups of 
$\Omega B\Rep(G)$.
This is discussed in detail in Section~\ref{Kdef-sec}, where $\Omega B\Rep(G)$ is compared to more highly structured models for deformation $K$--theory.
General linear   deformation $K$--theory is obtained by replacing $\U(n)$ by $\GL(n) = \GL_n (\bbC)$.
The (reduced, unitary) topological Atiyah--Segal map is a homomorphism
\begin{equation} \label{alpha-intro} \wt{\alpha}_* \co \rK_* (G) \defn \pi_* \Omega B\Rep(G) \maps \wt{K}^{-*} (BG),
\end{equation}
induced by a map of topological monoids arising from the natural maps 
\begin{equation}\label{B-intro}B\co \Hom(G, \U(n))\maps \Map_* (BG, B\U(n)).\end{equation}  
The details of this construction appear in Section~\ref{TAS-sec}.

Deformation $K$--theory has proven difficult to compute, and we do not expect any uniform description of its relationship to topological $K$--theory. This stands in contrast to 
similar functors such as algebraic $K$--theory of the group ring, or the $K$--theory of group $C^*$--algebras, where there are general conjectures describing this functors in terms of homological data.
We view this as a positive feature of the theory: it is subtle enough to capture delicate information about the group in question, so that when computation can be achieved, concrete consequences follow.

We consider several applications in this article.

\begin{itemize}

\item In Section~\ref{prev-sec}, we reinterpret a result from~\cite{RWY}  to show that rational surjectivity of $\wt{\alpha}_*$  in high dimensions implies the strong Novikov conjecture (Theorem~\ref{Nov}).  
Thus surjectivity of $\wt{\alpha_*}$ should be viewed as a very strong Novikov-type property.
Earlier work of the author implies that
surjectivity holds for surface groups (Theorem~\ref{TAS-2}).
 
\item In Section~\ref{H-sec}, we use Tyler Lawson's calculations of deformation $K$--theory for the 3--dimensional Heisenberg group $\pi_1 N^3$ to produce huge families of homotopy classes in the space of flat, unitary connections on bundles over the corresponding 3--dimensional Heisenberg manifold (Theorem~\ref{unbdd} and Corollary~\ref{fl-mon}).  This shows a marked difference between gauge theory in 2-- and 3--dimensions: over surfaces, homotopy in the space of flat connections is tightly controlled by Yang--Mills theory and complex geometric considerations, but in 3--dimensions the flood gates open. 

We note that for these results, the spectrum-level description of the topological Atiyah--Segal map is not really needed; one could instead take Theorem~\ref{TAS} as a definition.
 
\item In Section~\ref{image-sec}, we show that the topological Atiyah--Segal map is actually a map of $E_\infty$ rings, and consider some applications.
First, a result of S. P. Wang, combined with computational methods due to Lawson, allows us to calculate the deformation $K$--theory of groups satisfying Kazhdan's property (T), and to describe the topological Atiyah--Segal map for such groups (Propositions~\ref{K-T} and~\ref{T-prop}).  We then deduce that for property (T) groups admitting finite classifying spaces, the vector bundles  $E_\rho$ associated to
spherical families of unitary representations always represent torsion classes in  $\wt{K}^0 (S^m \cross BG)$ (Theorem~\ref{T-thm}). Finally, we show that the topological Atiyah--Segal map for the Heisenberg group fails to be surjective (Proposition~\ref{alpha-H}).

\end{itemize}

\section*{Acknowledgements}\label{ackref}
The results in Section~\ref{TAS-sec} arose from  conversations among the author, Willett, and Yu surrounding the work in~\cite{RWY}.
Proposition~\ref{tf} arose from discussions with Robert Lipshitz regarding work of his student Kristen Hendricks~\cite{Hendricks}.
The author thanks Baird, Lipshitz, Willett, and Yu, as well as Markus Szymik and Peter May, for helpful discussions. 
 
\section{Permutative categories arising from group actions}\label{act-sec}

In this section we introduce a framework for producing permutative categories, internal to the category $\Top$ of compactly generated topological spaces, from certain sequences of group actions. All topological notions, such as topological groups and group actions, should be interpreted in the category $\Top$.  

In what follows, we will assume familiarity with the notion of permutative categories, as defined in~\cite{May-perm}. 
We briefly recall the definition. A \e{permutative category} is a (small) internal category $\mathcal{C}$ in $\Top$ equipped with a distinguished unit object $*$ and a continuous functor $\oplus \co \mathcal{C}\cross \mathcal{C}\to \mathcal{C}$ which is \e{strictly} associative  and unital, and admits a natural commutativity isomorphism satisfying certain coherence requirements. Permutative functors between permutative categories are continuous functors that are appropriately compatible with the permutative structures. The category of permutative categories and permutative functors will be denoted $\PCat$.

The framework introduced here will be used in subsequent sections to give compatible descriptions of deformation $K$--theory and topological $K$--theory, facilitating the construction of the topological Atiyah--Segal map as a morphism of  spectra.  The proofs of the claims made in this section are all routine, and in fact relatively short, and will mostly be left to the reader.   
In Section~\ref{ring-sec}, we briefly explain how this theory can be enhanced to produce \e{bipermutative categories} and hence ring spectra.

We will use the following terminology regarding (topological) group actions: if $X$ is a $G$--space, $Y$ is an $H$--space, and $\phi \co G\to H$ is a (continuous) homomorphism,  then a map $f\co X\to Y$ is called $\phi$--equivariant, or equivariant with respect to $\phi$, if 
$$f(g\cdot x) = \phi(g) \cdot f(x)$$ 
for all $g\in G$, $x\in X$.  

\subsection{Action sequences}

The canonical example to keep in mind when reading the following definition is  the unitary (or general linear) groups acting on themselves by conjugation, with the usual matrix block sum operations (see Example~\ref{taut-add}).

\begin{definition}\label{p-action-seq}
A permutative (left) action sequence is a sextuple
$$\A = (I, \{G_i\}_{i\in I}, \{X_i\}_{i\in I}, *, \oplus, \{C_{i,j}\}_{i,j\in I}),$$ where:
\begin{itemize}
\item $I$ is a commutative monoid, with identity element $0$ and operation $+$;
\item Each $G_i$ is a topological group with identity element $e_i\in G_i$, and each $X_i$ is a left $G_i$--space;
\item $* = *_0\in X_0$ is a non-degenerate basepoint;
\item $\oplus = (\oplus^{\textrm{alg}}, \oplus^{\textrm{top}})$;
\item $\oplus^{\textrm{alg}}$ is an associative collection of homomorphisms
$$\oplus^{\textrm{alg}}_{i,j} \co G_i\cross G_j \to G_{i + j}, \,\,\, i,j\in I;$$
\item $\oplus^{\textrm{top}}$ is an associative collection of $\oplus^{\textrm{alg}}_{i,j}$--equivariant maps
$$\oplus^{\textrm{top}}_{i,j} \co X_i\cross X_j \to X_{i + j}, \,\,\, i,j\in I,$$
where equivariance refers to the product action of 
$G_i\cross G_j$ on $X_i\cross X_j$ (and the action of $G_{i+j}$ on $X_{i+j}$);
\item  For each $i, j\in I$, we have $C_{i,j}  \in G_{i+j}$.
\end{itemize}

From here on, we will usually simplify notation by writing $\oplus$ in place of $\oplus^{\textrm{alg}}_{i,j}$ or $\oplus^{\textrm{top}}_{i,j}$.

The elements $C_{i,j}$ are subject to the following further axioms for all $i, j, k\in I$.

\begin{itemize}

\item  $C_{i,j} \cdot (x_i \oplus x_j) = x_j \oplus x_i$ for each $x_i \in X_i$, $x_j\in X_j$

\item  $C_{i, 0} = C_{0,i} = e_i$;

\item  $C_{i,j}  C_{j,i}= e_{i+j}$;

\item  If $g_i\in G_i$ and $g_j\in G_j$, then 
$C_{i, j} (g_i \oplus g_j)  = (g_j \oplus g_i)  C_{i,j}$;

\item $(C_{i, k} \oplus e_j ) (e_i \oplus C_{j, k}) = C_{i+j, k}$.

\end{itemize}

Note that it is not necessary to assume $*_0$ is fixed by the action of $G_0$.

We refer to the operations $\oplus$ as the \e{monoidal}, or \e{additive}, structure of $\A$, and we refer to the elements $C_{i,j}$ as the \e{commutativity operators}.

It will be convenient to use the notation $X = \coprod_i X_i$, $G = \coprod_i G_i$, as well as to set $C = \{C_{i,j}\}$.
Then we can write an action sequence in the simplified notation
$\A = (I, G, X, *, \oplus, C)$.

A morphism of permutative action sequences 
$$\A = (I, G, X, *, \oplus, C) \maps 
\B =   (J, H, Y, *,  \oplus, D)$$ 
consists of a homomorphism of monoids $f\co I\to J$ together with group homomorphisms
$$\phi_i \co G_i \to H_{f(i)}$$
satisfying $\phi_{i+j} (C_{i,j}) = D_{f(i), f(j)}$,  
 and $\phi_i$--equivariant maps
$$\zeta_i\co X_i \to Y_{f(i)}$$
for each $i\in I$, such that
$\zeta_0 (*) = *$.
This defines the category $\PAct$ of permutative (left) action sequences. 

\end{definition}

We will always work with left actions, and we drop the adjective \e{left}.

\begin{construction}\label{act-cat} There is a functor
$$\T\co \PAct \maps \PCat,$$
defined as follows.  

Given an action sequence $\A = (I, G, X, *, \oplus, C)$,
the object space of $\T (\A)$ is simply 
$X =  \coprod_{i\in I} X_i$,
while the morphism space is
$\coprod_{i\in I} G_i \cross X_i$.
The domain of $(g, x)$ is $x$, the codomain is $g\cdot x$, and composition is given by
$$(h, g\cdot x) \circ (g, x) = (hg, x).$$
The operations $\oplus$ give rise to a continuous functor
$$\oplus \co \T(\A) \cross \T(\A) \maps \T(A),$$
which is (strictly) associative and has the object $*_0\in X_0$ as (strict) unit.  The commutativity isomorphisms are given by $(C_{i,j}, x_i \oplus x_j)$, and our axioms on the $C_{i,j}$ are exactly what is needed to make the coherence diagrams in $\T(\A)$ commute.
\end{construction}

We call $\T(\A)$ as the \e{translation category} of $\A$.
Note that $\T(\A)$ is in fact a groupoid.

\begin{example}\label{taut-add} The tautological (additive) unitary  permutative  action sequence is given by setting $I = \bbN$, with ordinary addition as the monoid operation, and setting $X_n = G_n = \U(n)$ for $n \in \bbN$.  We define $\U(0) = \{0\}$, the trivial group.
Here we view $\U(n)$ as a left $\U(n)$--space via \e{conjugation}, and we use the usual matrix block sum operation to define both $\oplus^{\textrm{alg}}$ and  $\oplus^{\textrm{top}}$, with $0\in \U(0)$ acting as the unit element. The commutativity operators are the (unitary) permutation matrices
\begin{equation}\label{Imn}I_{m,n} = \br{\begin{array}{cc} 0_{nm}  & I_n     \\
					 I_m  &  0_{mn}     \\     
	 \end{array}},\end{equation}
where $0_{pq}$ denotes the $p\cross q$ zero matrix.	 
	 The tautological additive general linear action sequence is defined similarly, by replacing $\U(n)$ by $\GL (n)$.
\end{example}

 \begin{definition} An (additive) \e{unitary  permutative action sequence} is one in which the underlying monoid is $\bbN$, with its usual addition, and we have $G_n = \U(n)$ and $C_{m,n} = I_{m,n}$   for all $m,n\in \bbN$.  
 Note that such sequences are completely determined by their topological data, that is,  the $\U(n)$--spaces $X_n$ (and the basepoint $x_0\in X_0$) together with the maps $\oplus^{\textrm{top}}$.
 \end{definition}

\begin{remark}\label{dep} The notion of a permutative action sequence can be generalized by allowing the elements $C_{i,j}$ to depend on $x_i\in X_i$ and $x_j\in X_j$ rather than just on $i,j\in I$, and a small modification again gives a functor from this larger category of sequences to $\PCat$.  Furthermore, there is no need to assume the $G_i$ are groups; monoids would suffice.\end{remark}

\subsection{The nerve of a permutative action sequence}\label{nerve-sec}

Consider a permutative action sequence $\A = (I, G, X, *, \oplus, C)$.
 The    continuous functor 
$$\oplus \co \T(\A) \cross \T(\A) \maps \T(\A)$$
  makes $|N_\cdot \T(\A)|$ into a topological monoid, since geometric realization commutes with products of simplicial spaces. 
  May's infinite loop space machine~\cite{May-perm} gives a functor $\mathbf{K}$ from $\PCat$ to the category of connective $\Omega$--spectra.  One key feature of this functor is that for each permutative category $\mcC$, the infinite loop space underlying the spectrum $\mathbf{K}(\mcC)$ is naturally weakly equivalent to $\Omega B |N_\cdot \mcC|$. 
  Our next goal is to give an explicit description of the monoid  $|N_\cdot \T(\A)|$.
  
For a space $X$ with a left action of a topological group $G$, the homotopy orbit space, or Borel construction, is the quotient space 
$$X_{hG} = (EG \cross X)/G,$$
where $EG$ is the geometric realization of the category $\overline{G}$, internal to $\Top$, with object space $G$ and morphism space $G\cross G$ (here $(g,h)$ is the unique morphism from $h$ to $g$, and $(g,h) \circ (h, k) = (g, k)$).  
Note that $\overline{G}$ admits a right action of $G$ (by functors), defined via right-multiplication in $G$.  This induces a right action of $G$ on $EG$, and now $G$ acts on $EG\cross X$ via $g\cdot (e, x) = (e\cdot g^{-1}, g\cdot x)$.
Let $BG$ denote the geometric realization of the category $\underline{G}$ (again internal to $\Top$) with a single object $*$ and morphism space $G$, and with composition $g\circ h = gh$.
When $G$ is a Lie group, the natural map $EG\to BG$ (induced by the functor sending a morphism $(g, h)$ in $\overline{G}$ to the morphism $gh^{-1}$ in  $\underline{G}$) is a universal principal $G$--bundle~\cite{Segal-class-SS}, and the natural map $X_{hG}\to BG$ is a fiber bundle with fiber $X$.

Given a   permutative action sequence $\A = (I, G, X, *, \oplus, C)$, we can form
$$\M(\A) \defn \coprod_{i\in I} (X_i)_{hG_i}.$$
The maps $\oplus_{i,j}^{\textrm{alg}}$  induce continuous functors
$\overline{G}_i \cross \overline{G}_j \maps \overline{G}_{i+j}$
and hence continuous maps
\begin{equation}\label{EG}EG_i \cross EG_j \maps EG_{i+j}.\end{equation}  
Since the maps $\oplus_{i,j}^{\textrm{alg}}$ are homomorphisms, the maps (\ref{EG}) are equivariant with respect to $\oplus \co G_i\cross G_j \to G_{i+j}$.
Together with the equivariant maps $\oplus \co X_i \cross X_j \to X_{i+j}$, these maps induce a map
$$\M(\A)\cross\M(\A) \maps \M(\A).$$
It is an exercise to check that this map 
makes $\M(\A)$ into a topological monoid with $[*, *_0]\in (X_0)_{hG_0}\subset \M(\A)$ as unit element, where   $*\in EG$ corresponds to the object in $\overline{G}$ represented by the identity of $G$.

\begin{proposition}\label{nerve} There is a natural homeomorphism of topological monoids
$$\M(\A)\maps |N_\cdot \T(\A)|.$$
\end{proposition}

\begin{proof} (Sketch) A special case of this statement is proven in \cite[Proposition 2.4]{Ramras-excision} using an argument due to Tyler Lawson.  That argument immediately generalizes to produce the desired map and to show that it is a continuous bijection.  The argument proceeds by viewing each side as the geometric realization of a simplicial space, and providing a map of simplicial spaces that is a homeomorphism on each level.
In~\cite{Ramras-excision}, an appeal to compactness was made to deduce continuity of the inverse maps on each level, but it is in fact a simple matter to write down  explicit formulas for these inverse maps, from which it is clear that the inverses are continuous.
\end{proof}

In~\cite{Ramras-gp-comp}, the notion of strongly anchored elements in a  (homotopy commutative) monoid $M$ was introduced, in order to give a concrete description of the homotopy groups $\pi_* \Omega BM$ of the homotopy group completion. We say that  $x\in M$ is strongly anchored if  there exists a homotopy between the maps 
$$(m,n) \goesto mn \textrm{\,\,\, and \,\,\,} (m,n)\goesto nm$$
 that is constant on $(x,x)$. 
We write $[X,Y]$ for the set of unbased homotopy classes of maps from $X$ to $Y$; if $Y = M$ is a topological monoid, this set acquires a monoid structure. (We will also write $[f]\in [X,Y]$ for the class of $f\co X\to Y$.)

\begin{theorem}[{\cite[Theorem 1.1, Proposition 2.8]{Ramras-gp-comp}}] $\label{htpy-gps}$ Let $M$ be a 
 topological monoid such that each $x\in M$ is strongly anchored.  Then for each $k \geq 0$, there is a natural isomorphism
\begin{equation} \label{Gamma-intro}  
\dfrac{\Gr[S^k, M]}{\Gr (\nu_k M)} \srm{\isom} \pi_k (\Omega BM),\end{equation}
where $\Gr$ denotes the Grothendieck group and $\nu_k M \leqs[S^k, M]$ is the  submonoid of nullhomotopic maps $S^k\to M$. 

In fact, for $k > 0$, the natural map 
$$\pi\co [S^k, M]\to  \dfrac{\Gr[S^k, M]}{\Gr (\nu_k M)}$$
is surjective, and $\pi( [\phi]) = 0$ if and only if there exists a constant map $c\co S^k\to M$ such that $\phi \bullet c$ is nullhomotopic $($where $\bullet$ denotes pointwise multiplication in $M$$)$.
\end{theorem}

Note  that for each $k\geqs 0$, there is a natural isomorphism between $\nu_k M$ and the monoid of path components $\pi_0 M$. 

One aspect of Theorem~\ref{htpy-gps} is that it describes the product in the group $\pi_k (\Omega BM)$ via the operation in the underlying monoid $M$. As one might guess, the proof of Theorem~\ref{htpy-gps} relies on a version of the Eckmann--Hilton Argument. 
Results of this type have a long history. Possibly the earliest such fact is that group operation in the fundamental group of a topological group, or more generally an $H$--space, can be defined either by loop concatenation or by point-wise multiplication. Another example is the fact that the group operation in the algebraic $K$--theory of a ring, defined for instance in terms of the homotopy groups of Quillen's $Q$--construction, is induced by the direct sum operation in the underlying category of finitely generated projective modules~\cite[p. 20]{Quillen}.

It will be convenient to define 
\[\rh_k M:= \dfrac{\Gr[S^k, M]}{\Gr (\nu_k M)}.\]
The above result tells us, in particular, that $\rh_k M$ is the quotient of $[S^k, M]$ by $\nu_k M$ in the category of abelian monoids. In particular, each element of $[S^k, M]$ is invertible modulo constant maps.

If $\A$ is a unitary or general linear permutative action sequence, then the argument in~\cite[Proof of Corollary 4.4]{Ramras-excision} shows that every element in the monoid $\M(\A)$ is $($strongly$)$ anchored. 

\begin{corollary}\label{unitary-htpy}
 If $\A$ is a unitary or general linear permutative action sequence, then 
 the natural map
$$ 
\rh_k \M(\A)\maps \Omega B \M(\A)$$
is an isomorphism for each $k\geqs 0$.
\end{corollary}

\begin{remark}\label{MA-htpy-comm}The commutativity operators $C_{i,j}$ induce a natural transformation between the functors
$\oplus \co \T(\A)\cross \T(\A) \to \T(A)$ and $\oplus \circ \tau$, where $\tau$ is the twist functor on the product category.  It then follows from basic categorical homotopy theory (\cite[Proposition 2.1]{Segal-class-SS}) that $|N. (\T(\A))| \isom \M(\A)$ is homotopy commutative.
However, this homotopy does not anchor elements.  The argument in~\cite[Proof of Corollary 4.4]{Ramras-excision} involves constructing different homotopies, specific to each element we wish to anchor.
 The main point in the proof is that for each $x\in X_n$, the matrix $I_{n,n}$ lies in the identity component of the stabilizer of $x^2$.  The fact that this requires no extra assumptions on the stabilizer appears to be a rather special feature of the unitary and general linear groups.
\end{remark}

\subsection{Bipermutative action sequences} \label{ring-sec}
We now extend the notion of permutative action sequence to a notion of \e{bipermutative action sequence}, in such a way that the translation category inherits the structure of a bipermutative category.  As discussed in Section~\ref{image-sec}, this gives a relatively efficient way of verifying that the topological Atiyah--Segal map is a map of ring spectra.

In short, a bipermutative action sequence is a pair of permutative action sequences, sharing the same indexing set $I$, the same spaces $X_i$, and the same groups $G_i$.  Additional coherence axioms relating the two permutative structures must hold, and these axioms imply that the two monoid structures on $I$ give it the structure of a rig, or a ``ring without negatives."   We give the details in Definition~\ref{bp-action-seq} below; the axioms are just direct translations of the axioms for bipermutative categories.   Maps of bipermutative action sequences are just maps that respect both permutative structures.

As an example, the Kronecker product of matrices endows the tautological unitary and general linear action sequences 
with a bipermutative structure.  The details are just an elaboration of the discussion in~\cite[VI \S5]{May-577}.  We note that some care must be taken when specifying the exact definition of Kronecker product, so that the coherence axioms hold.

\begin{definition}\label{bp-action-seq}
A bipermutative   action sequence is rig $R$ together with a pair of action sequences
$$((R, +),  G,  X, *_0, \oplus, C) \textrm{ and } ((R, \cdot),  G,  X, *_1, \otimes,  D)$$
sharing the same groups $G_r$ and the same $G_r$--spaces $X_r$ for all $r\in R$.  
Let $0\in R$ and $1\in R$ denote the additive and multiplicative identity elements of $R$, respectively.

These data must satisfy the following additional axioms for all $r, s, t, u\in R$
and all $x, y, z, w \in X$ and all $g, h, k, l \in G$:

\begin{itemize}
\item Zero Axioms:  $*_0 \otimes x  = *_0 = x \otimes *_0$ and $e_0 \otimes g  = e_0 = g  \otimes e_0$ (recall that $e_0 \in G_0$ is the identity element);

\item Right Distributivity Axioms: $(x\oplus y) \otimes z = (x  \otimes z) \oplus (y \otimes z)$ and $(g\oplus h) \otimes k = (g  \otimes k) \oplus (h \otimes k)$;

\item Coherence Axioms:
$C_{r,s} \otimes e_t = C_{r\cdot t, s \cdot t}$, and 
\begin{eqnarray*}(D_{t, r+s} \oplus D_{u, r+s})D_{r+s, t+u} =\hspace{2.75in}\\
\hspace{.5in} (e_{r\cdot t} \oplus C_{r\cdot u, s\cdot t} \oplus e_{s\cdot u}) \left[   \left[   (D_{t, r} \oplus D_{u, r}) D_{r, t+u} \right] \oplus \left[ (D_{t, s} \oplus D_{u, s})  D_{s, t+u} \right] \right].\end{eqnarray*}
\end{itemize}

We will sometimes denote these sequences in the simplified form
$$(R, G, X, *_0, *_1, \oplus, C, \otimes, D).$$

 A morphism of bipermutative action sequences
 $$(R, G, X, *_0, *_1, \oplus, C, \otimes, D) \maps (S, H, Y, *_0, *_1, \oplus, C', \otimes, D')$$
consists of a function $f\co R\to S$ that is a monoid homomorphism for both $+$ and $\cdot$, together with homomorphisms $\phi_r \co G_r\to H_{f(r)}$ and $\phi_r$--equivariant maps $X_r \to Y_{f(r)}$ (preserving both basepoints) for all $r\in R$.  The homomorphisms $\phi_r$ must satisfy $\phi_{s+t} (C_{s, t}) = C'_{f(s+t)}$ and $\phi_{s\cdot t} (D_{s, t}) = D'_{f(s\cdot t)}$.
 \end{definition}
 
 Bipermutative action sequence now form a category  $\BPAct$, and the translation category construction provides a functor from $\BPAct$ to the category of bipermutative categories internal to $\Top$, as defined in~\cite[Chapter VI]{May-577}.

  Multiplicative infinite loop space theory, as developed in~\cite{May-577, May-bipermutative, May-ring-space}, provides a functor taking a bipermutative category $\mcC$ to a (connective) $E_\infty$ ring spectrum $\bbK_\infty (\mcC)$.  Just as in the permutative case,  the underlying infinite loop space $\Omega^\infty \bbK_\infty (\mcC)$ is naturally weakly equivalent to the group completion $\Omega B \mcC$, where $B\mcC$ denotes the geometric realization of the bar construction applied to the nerve of $\mcC$, using its \e{additive} monoidal structure; the key step in the proof of this statement is~\cite[Theorem 9.3]{May-ring-space}.



\section{Deformation $K$--theory}$\label{Kdef-sec}$

Here explain how to view the unitary and general linear deformation $K$--theory spectra associated to finitely generated discrete groups in terms of action sequences.  

\begin{definition}
Given a discrete group $G$, let $\A(G)$ denote the unitary  permutative action sequence associated to the spaces $X_n = \Hom(G, \U(n))$, which we topologize as subsets of the mapping spaces $\Map(G, \U(n))$.  Note that $\Hom(G, \U(0))$ consists of a single point, which will serve as $*$.
We let the unitary groups act on these spaces by conjugation, and the block sum operations are induced by block sum of matrices.  Note that $\A(G)$ is contravariantly functorial in $G$.

The deformation $K$--theory spectrum  $\K (G)$ is the connective $\Omega$--spectrum associated to the permutative translation category $\T (\A(G))$, and we define 
$$\K_* (G) = \pi_* \K(G) \isom \pi_* \Omega^\infty \K (G).$$
\end{definition}

By Proposition~\ref{nerve}, we have a natural homeomorphism
$$\Omega^\infty \K(G) \isom \Omega B\left(\coprod_{n=0}^\infty \Hom(G, \U(n))_{\hU(n)}\right),$$
where the coproduct on the right has the topological monoid structure induced from the additive structure of $\A(G)$.

\begin{remark}\label{comp}
The spectra $\K(G)$ have been computed in a number of cases, and we briefly describe these results. 
When $G$ is the trivial group, $\K(G)$ is simply the connective $K$--theory spectrum $\ku$ (this follows, for instance, from~\cite[VIII \S2]{May-577}).
When $G$ is finite, $\K(G)$ decomposes as a wedge of copies of $\ku$, one for each irreducible complex representation of $G$~\cite[Chapter 6.3]{Lawson-thesis}. 
If $G$ is the fundamental group of a closed surface (or the circle), then $\K(G)$ is a wedge of suspensions of $\ku$, with $\beta_k (G)$ summands of the form $\Susp^k \ku$, where $\beta_k (G)$ is the $k$th Betti number~\cite[Theorem 6.1]{Ramras-stable-moduli}. That theorem also gives a slightly more complicated description for non-orientable surface groups. The behavior of deformation K-theory with respect to free products and direct products is also understood (see~\cite{Ramras-excision} and~\cite{Lawson-prod} respectively), allowing one to compute $\K(G)$ for finitely generated free and free abelian groups.
\end{remark}

Underlying deformation $K$--theory, we have the following topological monoids.

\begin{definition}
Given a discrete group $G$, define
$$\Rep (G) = \Rep(G, \U) \defn \coprod_{n=0}^\infty \Hom(G, \U(n))$$
and
$$\Rep(G)_{\hU} = \coprod_{n=0}^\infty \Hom(G, \U(n))_{\hU(n)}.$$
Block sum of matrices makes $\Rep(G)$ into a topological monoid, with the unique element in 
$\Hom(G,\U(0))$
 as the  identity.  Replacing $\U(n)$ by $\GL(n)$, we obtain the monoids $\Rep(G, \GL)$ and $\Rep(G, \GL)_{\hU}$.
\end{definition}

The main results of~\cite{Ramras-gp-comp} give a description of $\K_* (G)$ in 
terms of spherical families of representations.

\begin{theorem}[Theorem 7.3 in~\cite{Ramras-gp-comp}]$\label{Kdef}$ 
Let $G$ be a discrete group.
For each $m > 0$,
 there is are natural  isomorphisms 
$$ \K_m (G) \isom  \rh_m (\Rep(G)) \oplus K^{-m} (*),$$
$$\K_0 (G) \isom  \rh_0 (\Rep(G)) \isom \Gr (\pi_0 \Rep(G)).$$
Moreover, for $k>0$ the natural map of monoids
$$[S^k , \Rep (G)] \maps \rh_m (\Rep(G))$$
is \e{surjective}, with kernel
$$\{[\rho] \,:\, \exists \textrm{ a constant map } c\co S^k\to \Rep(G) \textrm{ such that } \rho\oplus c  \textrm{ is nullhomotopic} \}.$$
\end{theorem}

If $G$ is finitely generated, the analogous result holds in the general linear case.

It will also be helpful to consider a reduced form of deformation $K$--theory. 
The unitary and general linear cases of this discussion are identical.
 
 For each group $G$, the map $\{1\} \to G$ induces a map of spectra
\begin{equation}\label{q}q\co \K(G) \to \K(\{1\}),\end{equation}
which admits a splitting induced by the map $G\to \{1\}$. 
 As noted in Remark~\ref{comp}, we have $\K(\{1\})\heq \ku$.

\begin{definition}\label{rK}
We define $\rK (G)$ to be the homotopy fiber of the natural map 
$\K(G) \to \K(\{1\})$, and we set $\rK_* (G) = \pi_* \rK(G)$. 
\end{definition}

\begin{corollary}\label{rKdef}
There is a natural splitting
$$\K_m (G) \isom \rK_m (G) \oplus K^{-m} (*)$$
for each $m\geqs 0$, and for $m>0$   there are natural isomorphisms 
\begin{equation}\label{split-m>0} \rK_m (G) \isom \rh_m \Rep(G) \isom \pi_m \Omega B \Rep(G).\end{equation}

Additionally,   $\rK_0 (G)$ is naturally isomorphic to the quotient of 
$$\Gr (\pi_0 \Rep (G)) \isom \pi_0 \Omega B\Rep(G)$$ 
by the subgroup generated by the trivial 1-dimensional representation.
\end{corollary}

\begin{proof} The first statement is immediate from the fact that (\ref{q}) splits.
For $m>0$,~\cite[Proof of Theorem 6.5]{Ramras-gp-comp} shows that the sequence 
\begin{equation}\label{ses'} 0\maps \rh_m   \Rep (G) \srm{i_*}  \rh_m\left(\Rep(G)_{\hU}\right) \srm{q_*}
 \rh_m\left( \Rep(\{1\})_{\hU}\right) \maps 0.
 \end{equation}
 is split exact, which implies (\ref{split-m>0}).

When $m=0$, the splitting of $q\co \K(G) \to \K(\{1\})$ gives a natural isomorphism between $\rK_0 (G)$ and the cokernel of the map
$\K_0 (\{1\}) \to \K_0 (G)$,
whose image is generated by the trivial 1-dimensional representation.
\end{proof}

\section{Topological $K$--theory}$\label{K-top}$

In this section, $X = (X, x_0)$ will denote a (based) paracompact space having the homotopy type of compact Hausdorff space.  
The key case will be $X = B\pi_1 (K)$ with $K$ an aspherical finite CW complex.
We will define a permutative category whose homotopy groups agree with the complex topological $K$--theory of $X$.  Our construction, and the subsequent discussion, is designed to mirror the construction of deformation $K$--theory in the previous section.  This will facilitate our construction and analysis of the topological Atiyah--Segal map in the next section.  
As in the previous section, there is both a unitary and a general linear version of the constructions given here.  We focus on the unitary case; the general linear case is completely analogous (more so than for deformation $K$--theory).

\begin{definition} \label{map-act} As before, let $B\U(n)$ denote the geometric realization of the one-object category $\underline{\U(n)}$. 
Then the (left) conjugation action of $\U(n)$ on itself induces an action, by continuous functors, of $\U(n)$ on the category $\underline{\U(n)}$, and hence an action of $\U(n)$ on $B\U(n)$.  This in turn induces an action of $\U(n)$ on the based mapping space
$\Map_*(X, B\U(n))$.  Throughout this section, we will view $\Map_*(X, B\U(n))$ as a  $\U(n)$--space under this action.

The block sum  operations $\oplus \co \U(m) \cross \U(n) \to \U(m+n)$ are homomorphisms, and hence induce continuous functors $\underline{\U(m)} \cross \underline{\U(n)}\to\underline{\U(m+n)}$, maps 
$$B\U(m) \cross B\U(n) \maps B\U(m+n),$$
and equivariant maps
$$\oplus \co \Map_*(X, B\U(m))\cross \Map_*(X, B\U(n))\maps \Map_*(X, B\U(m+n)).$$
Since block sum is equivariant with respect to conjugation (that is, $(CAC^{-1}) \oplus (DBD^{-1}) = (C\oplus D)(A\oplus B)(C\oplus D)^{-1}$), functoriality implies
that this data gives a unitary permutative action sequence $\A_K (X)$ with $n$th space $\Map_*(X, B\U(n))$.  

Note that when $X = *$, we recover the unitary permutative action sequence whose associated spectrum is $\ku$ (since we are taking homotopy orbits of $BU(n)$ with respect to the conjugation action of $U(n)$).

Let $\mathcal{C}_K (X) = \T( \A_K (X))$ be the  translation category of $\A_K (X)$, and let $\mathcal{K} (X)$ denote the associated  spectrum.  Let $\wt{\mcK} (X)$ denote the homotopy fiber of the natural map $\mcK (X) \to \mcK(*)$.
Finally, set $\mcK_* (X) = \pi_* \mathcal{K} (X)$ and  $\wt{\mcK}_* (X) = \pi_* \wt{\mathcal{K}} (X)$.
\end{definition}

We have the following consequence of Propositions~\ref{nerve} and~\ref{unitary-htpy}.

\begin{corollary}\label{mcK} The geometric realization of the nerve of $\mathcal{C}_K (X)$ is isomorphic, as a topological monoid, to the topological monoid
$$\V(X)_{\hU} \defn \coprod_n \Map_*(X, B\U(n))_{\hU(n)},$$
and for each $m\geqs 0$, there are natural isomorphisms
$$ 
\rh_m \left(\V (X)_{\hU} \right) \srm{\isom} \pi_m \Omega B \left(\V (X)_{\hU} \right) \isom \mcK_m (X).$$
\end{corollary}

Our goal in this section is to compare the homotopy groups $\mcK_* (X)$ with the (complex) topological $K$--theory of $X$ (for $*\geqs 0$).

 We need to specify a definition of  topological $K$--theory.  Note that the Group Completion Theorem gives a natural homotopy equivalence
 $$\bbZ \cross BU \maps  \Omega B \left(\coprod_n B\U(n)\right),$$
 where $BU = \colim_n B\U(n)$, the colimit being formed with respect to the maps induced by block sum with the identity $I_1 \in \U(1)$. 
 
 For based spaces $Y$ and $Z$, let $\langle Y, Z\rangle$ be the set of based homotopy classes of based maps, and let  $\langle f\rangle$ be the based homotopy class of $f$.
 
 \begin{definition} For $m\geqs 0$, the \e{reduced} topological $K$--theory of $X$ is
$$\wt{K}^{-m} (X) = \wt{K}^0 (S^m \sm X) = \langle S^m \sm X,  \bbZ \cross BU \rangle$$
and the unreduced $K$--theory of $X$ is 
$$K^{-m} (X) = \wt{K}^{-m} (X) \oplus  K^{-m} (*) =  \wt{K}^{-m} (X) \oplus  \pi_m (\ku).$$

Maps $X\to Y$ induce maps on both reduced an unreduced $K$--theory (in the latter case, all maps act as the identity on $K^{-m} (*)$).
\end{definition}

Note that by the Group Completion Theorem, we have a natural zig-zag of homotopy equivalences connecting $\bbZ\cross BU$ and $\Omega B (\coprod_n BU(n))$, so one may replace $\bbZ\cross BU$ by 
$\Omega B (\coprod_n BU(n))$ in the above definition.
 
We will need to consider another topological monoid related to $\mcK (X)$.

\begin{definition} Define
$\V (X) := \coprod_n \Map_* (X, B\U(n))$, 
and equip $\V (X)$ with the monoid structure induced by the block sum operations on $\{B\U(n)\}_n$.  
\end{definition}

We now have the following analogue of the results from Section~\ref{Kdef-sec}.

 \begin{corollary}\label{rK-cor}
Let $(X, x_0)$ be a based, path-connected CW complex  homotopy equivalent to a finite CW complex.
For each $m\geqs 0$, there is a natural splitting
\begin{equation}\label{K-split} \mcK_m (X) \isom \wt{\mcK}_m (X) \oplus \mcK_m (*) = \wt{\mcK}_m (X) \oplus \pi_m \ku\end{equation}
and a natural
isomorphism
$$\wt{\mcK}_m (X) \isom \wt{K}^{-m} (X).$$

Moreover,  for  $m> 0$ there are natural isomorphisms 
$$ \wt{\mcK}_m (X) \isom \rh_m \V(X) \isom \pi_m \Omega B \V(X),$$
while for $m=0$, $\wt{\mcK}_0 (X) \isom \wt{K}^0 (X)$ is naturally isomorphic to the quotient of 
$$\Gr (\pi_0  \V(X)) \isom \pi_0 \Omega B\V(X)$$ 
by the subgroup generated by the class of nullhomotopic maps  $X\to B\U(1)$.
\end{corollary}

\begin{proof} 
The splitting (\ref{K-split}) is immediate since the inclusion $\{x_0\} \injects X$ splits.
Next, consider the sequence of groups
\begin{equation}\label{split-seq}\rh_m (\V (X)) \maps \rh_m (\V (X)_{hU}) = \mcK_m (X) \maps \rh_m \left(\coprod_n (BU(n))\right) = \mcK_m (*)
\end{equation}
induced by the maps
\begin{equation}\label{split-seq2}\Map_*(X, BU(n)) \maps \Map_*(X, BU(n))_{hU(n)} \maps BU(n).
\end{equation}
Using the fact that the sequences (\ref{split-seq2}) are (split) fibrations, one can check directly that 
 (\ref{split-seq}) is a (split) short exact sequence for each $m>0$. By naturality, it follows that
 $\rh_m (\V (X)) \isom \wt{\mcK}_m (X)$ for $m>0$. 
 We will return to the analogous description of $\wt{\mcK}_0 (X)$ at the end of the proof.

Next we show that for $m>0$, there is an isomorphism
$$\wt{K}^{-m} (X) = \langle S^m \wedge X, \bbZ \cross BU \rangle \isom \langle S^m, \Map_*(X, BU)\rangle 
\stackrel[\isom]{\Phi}{\longrightarrow} \rh_m \V (X).$$
To define the map $\Phi$, note that our hypotheses on $X$ guarantee that every map $f\co S^m\sm X\to BU$ factors, up to homotopy, through $BU(n)$ for some $n$, yielding a well-defined element of $ \rh_m \V (X)$.
To see that $\Phi$ is surjective, note that our hypotheses guarantee that every vector bundle over $X$ is a direct summand of a trivial bundle, and hence every class in $ \rh_m \V (X)$ is represented by a map 
$$S^m \to \Map_* (X, BU(n))$$
 whose value the basepoint in $S^m$ is nullhomotopic. Since the basepoint of $S^m$ is non-degenerate, we can in fact obtain a representative lying in $$\Map_* (S^m, \Map_*(X, BU(n))),$$ as desired. To prove  injectivity of $\Phi$, say $\Phi (\langle f \rangle) = 0$ for some 
 $$f\in \Map_* (S^m, \Map_*(X, BU(n))).$$ 
 Note that it since the basepoint of $X$ is non-degenerate, will suffice to show that $f$ is nullhomotopic in the \e{unbased} sense (see~\cite[Section 4A]{Hatcher}, for instance). Since $\Phi (\langle f \rangle) = 0$, there exist $h, g\in \Map_*(X, BU(k))$ (for some $k$) such that
$$[f\oplus c_g] = [c_h],$$ 
where $c_g$ and $c_h$ denote the constant maps out of $S^m$ with values $h$ and $g$. Now, there exists $g'\in \Map_*(X, BU(l))$ such that $g\oplus g'$ is nullhomotopic, so 
we have 
$$[f] = [f\oplus c_g \oplus c_{g'}] = [c_{h+g'}]$$ 
in $[S^m, \Map_* (X, BU)]$, 
and hence $f$ is nullhomotopic as a map into $\Map_* (X, BU)$.

 For $m=0$, connectedness of $U(n)$ yields bijections 
 $$\pi_0 \Map_* (X, BU(n))_{hU(n)} \isom \pi_0 \Map_* (X, BU(n)) \isom  \langle X, BU(n)\rangle
 \isom [X, BU(n)]$$
 so 
 $$\mcK_0 (X) = \pi_0 \V(X)_{hU} \isom \Gr \left( \pi_0 \V(X)_{hU}\right) \isom \Gr \left(  \pi_0 \V(X)\right)  = \Gr [X, \coprod_n BU(n)].$$
Now, 
$$\Gr [X, \coprod_n BU(n)] \isom \colim ([X, \coprod_n BU(n)] \to [X, \coprod_n BU(n)] \to \cdots),$$ 
where the colimit is taken with respect to addition by the basepoint in $BU(1)$.
This colimit is precisely $\bbZ \cross \colim_n [X, BU(n)] \isom \bbZ \cross \langle X, BU\rangle \isom K^0 (X)$.

The splitting (\ref{K-split}) identifies $\wt{\mcK}_0 (X)$ with the cokernel of $\mcK_0 (*) \to \mcK_0 (X)$.
By the discussion above, this agrees with cokernel of  $\Gr \left(  \pi_0 \V(*)\right) \to \Gr \left(  \pi_0 \V(X)\right)$. The claimed description of $\wt{\mcK}_0 (X)$ now follows from the fact that $\V(*)$ is generated by the constant map from $*$ to the basepoint of $BU(1)$.
 \end{proof}

\begin{remark}  Corollary~\ref{rK-cor} requires the assumption that $X$ is path-connected.
For instance, Corollary~\ref{mcK} gives 
$$\mcK_0 (S^0) \isom \Gr\left( \coprod_n \Map_* (S^0, BU(n))_{hU(n)} \right),$$ 
and since $BU(n)$ is simply-connected, each set 
$$\pi_0 \Map_* (S^0, BU(n))_{hU(n)} \isom \pi_0 \Map_* (S^0, BU(n))$$
is a singleton. So $\mcK_0 (S^0) = \bbZ$, whereas $K^0 (S^0) = \bbZ \oplus \bbZ$.
\end{remark}

 
\section{The topological Atiyah--Segal map}$\label{TAS-sec}$

Let $G$ be a discrete group whose classifying space $BG$ has the homotopy type of a finite CW complex.
We now define reduced and unreduced versions of the topological Atiyah--Segal map, which relates the deformation $K$--theory of $G$ to the topological $K$--theory of $BG$.  
The unitary and general linear discussions are completely parallel, so we focus on the unitary case.

The classical    Atiyah--Segal map associates to each representation $\rho\co G\to \U(n)$ the $K$--theory class represented by the vector bundle $E_\rho = (EG\cross \bbC^n)/\pi_1 G  \to BG$, where $\pi_1 G$ acts via $\gamma\cdot (e, v) = (e\cdot \gamma^{-1}, \rho(\gamma) v)$.
We will see that in dimension zero, the classical    Atiyah--Segal map factors as
$$R[G] = \Gr (\Rep(G)^\delta) \maps \Gr(\pi_0 \Rep(G)) \srm{\alpha_0}   K^0 (BG),$$
where $\Rep(G)^\delta$ is the discrete monoid underlying the topological monoid $\Rep(G)$, and $\alpha_0$ is the topological Atiyah--Segal map (in dimension 0).

The simplicial classifying space functor $B$ induces continuous, $\U(n)$--equivariant maps
\begin{equation} \label{B}
\xymatrix@R=2pt{ B = B_n \co \Hom(G, \U(n)) \ar[r] & \Map_* (BG, B\U(n)) \\
 \hspace{.5in}\rho\ar@{|->}[r] & B\rho}
\end{equation}
which combine to give a map between the associated unitary permutative action sequences.  Recall that the   spectra associated to these action sequences are $\K(G)$ and $\mcK (BG)$, respectively, and the homotopy groups of the latter are the complex $K$--theory groups of $BG$.

\begin{definition}\label{TAS-def}
The \e{topological Atiyah--Segal map} 
$$\alpha = \alpha^G \co \K(G) \maps \mcK (BG)$$
is the map of  spectra induced by the above map of permutative action sequences.  The reduced topological Atiyah--Segal map 
$$\wt{\alpha} \co \rK (G) \maps \wt{\mcK} (BG)$$
is the induced map 
$$\hofib \left(\K(G) \to \K(\{1\})\right) \maps \hofib \left( \mcK(BG) \to \mcK(*)\right).$$
Note that we have $\K(\{1\}) = \mcK(*) \heq \ku$, and $\alpha^{\{1\}}$ is the identity map.
\end{definition}

The results in the previous sections combine to give the following descriptions of the  induced maps $\alpha_*$ and $\wt{\alpha}_*$ on homotopy groups.

\begin{corollary} 
For $m\geqs 0$,  the   topological Atiyah--Segal map is naturally isomorphic to the maps 
$$\pi_m \Omega B\Rep(G)_{\hU} \maps \pi_m \Omega B \V (BG)_{\hU}$$
and 
$$\rh_m \Rep(G)_{\hU} \maps \rh_m \Omega B \V (BG)_{\hU}$$
induced by the simplicial classifying space functor $B$, and there is a natural splitting
\begin{equation}\label{alpha-split}\alpha_* = \wt{\alpha}_* \oplus \mathrm{Id}_{\pi_* \ku}.
\end{equation}

Moreover, for $m > 0$,  $\wt{\alpha}_m$ is naturally isomorphic to the maps
$$\pi_m \Omega B\Rep(G)  \maps \pi_m \Omega B \V (BG)$$
and 
$$\rh_m \Rep(G) \maps \rh_m   \V (BG)$$
induced by $B$.
\end{corollary}

Our goal in this section is to give an explicit description of the topological Atiyah--Segal map in terms of vector bundles.
Consider the diagram
\begin{equation}\label{TAS-diag}
\xymatrix{ 		\rh_m  \Rep (G) \ar[r]^{\wt{\alpha}_m} 
				\ar@{-->}[dr] 
			& \rh_m \V(BG)  
			\\
			& \pi_m \Map_* (BG, B\U) \ar[u]_-\isom^-\Psi  
		&\save[] *\txt<8pc>{
 \,\,\,\,\,$\isom \wt{K}^{-m} (BG)$.}
\restore		
		}
\end{equation}		
where  
$\Psi$ is defined by $\Psi (\langle f\rangle) = [f]$.  We will describe the diagonal map $\Psi^{-1} \circ \wt{\alpha}_m$.

By Theorem~\ref{Kdef},  
each class in $\rh_m \Rep(G)\isom \K_m (G)$ has a representative of the form $[\rho]$, with $\rho$ a family of representations $\rho\co S^m \to \Rep (G)$.  
Let
$E_\rho$ be the right principal $\U(n)$--bundle over $S^m \cross BG$
defined by
\begin{equation*}  
\xymatrix@R=2pt{E_\rho = \left(S^m \cross EG \cross \U(n)\right)/G \ar[r] & S^m \cross BG \\
 [z, e, A] \ar@{|->}[r] & (z, q(e)),}
\end{equation*}
where $q\co EG\to BG$ is the bundle projection and $g\in G$ acts via 
$$g \cdot (z, x, A) \defn (z, x\cdot g^{-1}, (\rho (z) (g)) A).$$
Basic properties of this construction are reviewed in~\cite[Section 3]{Baird-Ramras-smoothing}.

We will use $1\in S^0 \subset S^m$ as the basepoint of $S^m$, and for any family 
$$\rho\co S^m\to \Hom(G, \U(n)),$$ 
we let $\wt{\rho(1)} \co S^m\to \Hom(G, \U(n))$ denote the constant family at 
$\rho(1) \co G \to \U(n)$.

For  based CW complexes $X_1$ and $X_2$, the long exact sequence in $K$--theory for the pair $(X_1\cross X_2, X_1 \vee X_2)$ yields a (naturally) split short exact sequence 
$$0\maps \wt{K}^0 (X_1\sm X_2) \srm{\pi^*} \wt{K}^0 (X_1\cross X_2)\srm{i^*} \wt{K}^0 (X_1\vee X_2)\maps 0.$$
If $\rho$ is an $S^m$--family ($m>0$),
the bundle $E_\rho \to S^m \cross BG$ is trivial when restricted to $S^m \cross \{x\}$ (for each point $\{x\}\in BG$): indeed, each point $\wt{x}\in q^{-1} (x) \subset EG$ gives rise to a continuous section $z\goesto [z, \wt{x}, I]$ of the restricted bundle.  Thus we have  
$$E_\rho|_{S^m \vee BG} \isom E_{\wt{\rho(1)}}|_{S^m \vee BG},$$ 
and hence
$$[E_\rho] - [E_{\wt{\rho(1)}}] \in \ker (i^*) = \Img (\pi^*).$$ 
Since $\pi^*$ is injective, it follows that the class $[E_\rho] - [E_{\wt{\rho(1)}}]$ has a well-defined pre-image under  $\pi^*$, which we will denote by
\begin{equation}\label{pi-inv}
(\pi^*)^{-1} ([E_\rho] - [E_{\wt{\rho(1)}}]) \in \wt{K}^0 (S^m\sm BG) = \wt{K}^{-m} (BG).
\end{equation}
By~\cite[Lemma 3.1(2)]{Baird-Ramras-smoothing}, the bundles $E_\rho$ and $E_{\wt{\rho(1)}}$ only depend (up to isomorphism) on the unbased homotopy class of $\rho$.  Hence the class (\ref{pi-inv})  depends only the unbased homotopy class of $\rho$.
Note that 
if $\rho$ is constant, then $E_\rho =   E_{\wt{\rho(1)}}$, so in this case the class (\ref{pi-inv}) is trivial.  

With this understood, we have the following explicit description of  $\wt{\alpha}_*$ (or more precisely, of the map $\Psi^{-1} \circ \wt{\alpha}_*$ in Diagram (\ref{TAS-diag})).

\begin{theorem} \label{TAS} Let $G$ be a group whose classifying space $BG$ is homotopy equivalent to a finite CW complex.  
Then for $m\geqs 1$, the reduced topological Atiyah--Segal map,
viewed as a map
$$\rh_m \Rep (G)  \maps \wt{K}^{-m} (BG)$$
via the diagram $(\ref{TAS-diag})$, has the form
\begin{equation}\label{wtalpha} [\rho] \goesto (\pi^*)^{-1} \left( [E_\rho] - [E_{\wt{\rho(1)}}]\right).\end{equation}

When $m=0$, the map
$$\alpha_0\co \K_0 (G)\isom \Gr (\pi_0 \Rep(G)) \maps  K^0 (BG)$$
is given by $\alpha_0 ([\rho]) = [E_\rho]$.
 \end{theorem}
 
 Note that the statement for $m=0$ shows that  the classical Atiyah--Segal map  factors through $\alpha_0$, as claimed earlier.
 
 \begin{proof}  We assume  $m>0$; the proof for $m=0$ is similar but simpler.  By definition, $\wt{\alpha}_* ([\rho]) = [B\circ \rho]$, where $B$ is the map (\ref{B}).  
Let $f\co S^m\sm BG\to B\U(n)$ be a map classifying $(\pi^*)^{-1} \left( [E_\rho] - [E_{\wt{\rho(1)}}]\right)$.  To prove the proposition, we need to show that the adjoint map $f^\vee \co S^m \to \Map_*(BG, B\U(n))$ satisfies 
$$\Psi(\la f^\vee \ra) = [B\circ \rho]$$
 in $\rh_m \V (BG)$.  (Recall that
$\Psi (\la f^\vee \ra)$ is simply $[  f^\vee  ]$.)

By choice of $f$, the composite $f\circ \pi$
 classifies $[E_\rho] - [E_{\wt{\rho(1)}}]$, and
by~\cite[Lemma 4.1]{Baird-Ramras-smoothing},
 if $c$ is the constant map $S^m\to \Map_* (BG, B\U(p))$ with image $B(\rho(1))$, 
 then $c^\vee \circ \pi$
 classifies $[E_{\wt{\rho(1)}}]$ (here $c^\vee$ is the adjoint of $c$),
 while  $(B\circ \rho)^\vee \circ \pi$ classifies  $[E_\rho]$.  Hence the maps
 $$(f\circ \pi) \oplus ( c^\vee \circ \pi  ) = (f\oplus c^\vee) \circ \pi \textrm{ \,\,\, and \,\,\,} (B\circ \rho)^\vee \circ \pi$$ 
 represent the same class in $\wt{K}^0 (S^m\sm BG)$.  
Since $\pi^*$ is injective, 
it follows that $f\oplus c^\vee$ and $(B\circ \rho)^\vee $ are homotopic as maps $S^m\sm BG\to B\U(N)$ (for sufficiently large $N$), and consequently $f^\vee \oplus c$ and $B\circ \rho$ are homotopic as maps $S^m\to \Map_*(BG, B\U(N))$.  Since $c$ is constant, it follows that $[B\circ\rho] = [f^\vee] =  \Psi (\la f^\vee \ra)$ in $\rh_m \V (BG)$. 
\end{proof}

One may replace $\U(n)$ by $\GL (n)$ throughout the preceding discussion, yielding a general linear    version $\alpha^{\GL}$ of the topological Atiyah--Segal map.  We note that the unitary   topological Atiyah--Segal map factors through this general linear version, which leads to the following natural question.

\begin{question} Does there exist a group $G$ for which the image of $\alpha^{\GL}_*$ is strictly larger than the image of $\alpha_*$?
\end{question}

\section{Relations with previous work}\label{prev-sec}

In this section, we reinterpret some of the main results from~\cite{Baird-Ramras-smoothing},~\cite{RWY}, and~\cite{Ramras-stable-moduli} in terms of the topological Atiyah--Segal map.

\subsection{Bounds on the image of $\alpha_*$}

We now show that $\wt{\alpha}_*$ fails to be surjective in dimensions $\qcd (G) - 2k$ ($k>0$), where $\qcd (G)$ is the largest number $n$ for which $H^n (G; \bbQ)$ is non-zero (Theorem~\ref{BR}).
This low-dimensional failure is closely analogous to the low-dimensional failure of the Quillen--Lichtenbaum conjectures (in the form discussed in~\cite{OR}, for instance), which the relate algebraic $K$--theory of a field $k$ to its \'{e}tale K--theory in dimensions greater than the virtual cohomological dimension of $k$ minus 2.

\begin{theorem} $\label{BR}$ The image of
$\alpha_* \co \K_m (G) \to K^{-m} (BG)$
$($and, in fact, of $\alpha^{\GL}_*$$)$ 
has rank at most $\beta_m(G) + \beta_{m-2} (G) + \cdots$, where $\beta_i (G)$ is the rank of $H_i (BG; \bbZ)$.
Hence if $\beta_d (G)\neq 0$, then the maps $\alpha_{d-2}, \alpha_{d-4}, \ldots$ are not surjective. 
\end{theorem}

\begin{proof}  By Theorem~\ref{TAS} and~\cite[Theorem 3.5]{Baird-Ramras-smoothing}, the image of $\alpha_m$ lies in the subgroup of $K^{-m} (BG)$ on which the Chern classes $c_{m+i}$, $i=1, 2, \cdots$, vanish rationally.  It follows from~\cite[Lemma 4.2]{Baird-Ramras-smoothing} that the rank of this subgroup is given by the above sum.  The last statement follows from the fact that  the Chern character is a rational isomorphism, which implies that the rank of $K^{-m} (BG)$ is equal to the sum of all the Betti numbers of $BG$ in dimensions equivalent to $m$ modulo 2.
\end{proof}

\subsection{Relation with the Novikov conjecture}

A group $G$ satisfies the \e{strong Novikov conjecture} if the analytical assembly map, from the $K$--homology of $BG$ to the $K$--theory of the maximal $C^*$--algebra of $G$,  is injective after tensoring with $\bbQ$.  For background on this conjecture, see~\cite{RWY} and the references therein.

\begin{theorem}$\label{Nov}$ If $G$ is a group such that $BG$ is homotopy equivalent to a finite CW complex, and the unitary topological Atiyah--Segal map $\alpha_m$ is rationally surjective for all sufficiently large $m$, then $G$ satisfies the strong Novikov conjecture.
\end{theorem}

\begin{proof} Surjectivity of $\alpha_m$ is equivalent to surjectivity of $\wt{\alpha}_m$, and 
surjectivity of $\wt{\alpha}_m$ implies that for sufficiently large $n$, every element in 
$[S^m,  \Map_* (BG, B\U(n))]$ has the form $[B\circ \rho]$ for some $\rho$.  It now follows from~\cite[Theorem 3.16]{RWY} (or rather from the proof of that result) that $G$ lies in the class of \e{flatly detectable groups}; all such groups satisfy the strong Novikov conjecture by~\cite[Corollary 4.3]{RWY}.
\end{proof}

\subsection{Surface groups}

We now translate  the results in~\cite{Ramras-stable-moduli} into information about the topological Atiyah--Segal map when $G = \pi_1 (\Sigma)$ is the fundamental group of a compact, aspherical surface $\Sigma$.  Note that we allow $\Sigma$ to have boundary, in which case $\pi_1 (\Sigma)$ is a finitely generated free group.
We begin by recalling
a result from~\cite{Ramras-stable-moduli}, which was proven using Morse theory for the Yang--Mills functional.

\begin{theorem}[\cite{Ramras-stable-moduli}, Theorem 3.4]$\label{B-thm}$ If $\Sigma$ is a compact aspherical surface, possibly with boundary, then for each $M\geqs 0$, there exists $N$ such that for every $n>N$ the natural map 
$$B_* \co \pi_m \Hom(\pi_1 \Sigma, \U(n)) \maps \pi_m \Map_* (B\pi_1 \Sigma, B\U(n))$$
induces an isomorphism on homotopy groups in dimensions $1 \leqs m \leqs M$ for all choices of compatible basepoints.
If $\Sigma$ is non-orientable, or has non-empty boundary, then this statement holds for $0\leqs m\leqs M$.
\end{theorem}

\begin{theorem} $\label{TAS-2}$ If $\Sigma$ is a compact aspherical surface, possibly with boundary, then the topological Atiyah--Segal map
$$\alpha_m \co \K_m (\pi_1 \Sigma) \maps K^{-m} (\Sigma)$$
is an isomorphism for $m\geqs 1$.  If $\Sigma$ is non-orientable, or  has non-empty boundary, then $\alpha_0$ is an isomorphism as well.
\end{theorem}
\begin{proof}  Note that it suffices to prove that $\wt{\alpha}_m$ is an isomorphism.
 We work in the case  $m>0$; the same reasoning will apply when $m=0$ and $\Sigma$ is non-orientable or has non-empty boundary, using the last part of Theorem~\ref{B-thm}.

We prove injectivity; the proof of surjectivity is similar but simpler. 
By Theorem~\ref{Kdef}, each element in the kernel of $\wt{\alpha}_m$ has the form
$[\rho]$ for some $$\rho\co S^m \to \Hom(\pi_1 \Sigma, \U(n))$$ satisfying 
 \begin{equation}\label{inj}(B\circ \rho) \oplus d \heq c \end{equation}
 for some constant maps $c, d\co S^m \to \V (B\pi_1 \Sigma)$.  
 In fact, by~\cite[Corollaries 4.11 and 4.12]{Ramras-surface} we may assume without loss of generality that $d = BI_p$ for some $p$, where $I_p$ denotes the constant map $S^m \to \Hom(\pi_1 \Sigma, \U(p))$ with image the trivial representation. Indeed, for orientable surfaces the homomorphism spaces are all path connected, while in the non-orientable case each homomorphism space has two path components, and adding representations from the non-trivial components always yields a representation in the trivial component, so (for instance) we may simply replace $d$ by $d\oplus d$ if necessary.
 
Equation (\ref{inj}) now implies that  $B\circ (\rho \oplus I_p)$ is nullhomotopic, and the injectivity portion of  Theorem~\ref{B-thm}  implies that $\rho\oplus I_p$ must be nullhomotopic. Hence $[\rho] = 0$ in $\rh_m (\Rep(\pi_1 \Sigma))$, as desired.
\end{proof}

\section{Families of flat connections over the Heisenberg manifold}\label{H-sec}

In this section we study flat, unitary connections on   complex vector bundles  over the 3--dimensional Heisenberg manifold by combining the main results of this paper with computations due to Tyler Lawson.
We begin with a review of the definition and  the basic properties of this manifold.

\subsection{Background.}

The discrete 3--dimensional Heisenberg group $H$ is the group of $3\cross 3$ upper triangular integer matrices  under ordinary matrix multiplication.  This group sits as a (discrete) subgroup of the real Heisenberg group $H_\bbR$, which consists of all real upper triangular matrices.  We will identify the real Heisenberg group with $\bbR^3$ via the function
$$\br{\begin{array}{rrrr}  1 & x  & z     \\
					  0 &  1& y     \\
					   0& 0  &1       
	 \end{array}}  \goesto (x,y,z)$$
(note, though, that we are using matrix multiplication to define the operation in $H_\bbR$, not addition in the vector space $\bbR^3$).  The Heisenberg manifold is defined by 
$$N^3 = \bbR^3/H,$$
where $H$ acts on $\bbR^3 \isom H_\bbR$ by (left) multiplication.

It is an elementary exercise to check that $N^3$ is Hausdorff, and that the map $\bbR^3\to N^3$ is a covering.  In particular, this means $N^3$ is an aspherical 3-dimensional manifold with fundamental group $H$ (and hence $N^3\heq BH$), and $N^3$ is orientable since the action of $H$ on $\bbR^3$ is orientation-preserving.  Moreover, $N^3$ is compact; this follows, for instance, from the fact that each closed unit cube in  $\bbR^3$ surjects onto $N^3$.
In fact, $N^3$ is a circle bundle over $\bbR^2/\bbZ^2$.
Indeed, consider the mapping
$$N^3 = \bbR^3/H \srm{q} \bbR^2/\bbZ^2$$
given by sending $[(x,y,z)]$ to $[(x,y)]$.  It is elementary to check that this map is a fiber bundle with circle fibers; indeed, for each $[(x,y)]\in \bbR^2/\bbZ^2$, there exists $\epsilon>0$ such that the mapping
$$\xymatrix@R=1pt{  [x-\epsilon, x+\epsilon] \cross [y-\epsilon, y+\epsilon] \cross \bbR/\bbZ \ar[r] & N^3\\
(x', y', [z]) \ar@{|->}[r] & [(x', y', z)]}
$$
is a homeomorphism onto 
$$q^{-1} \left(\pi( [x-\epsilon, x+\epsilon] \cross [y-\epsilon, y+\epsilon])\right),$$
where $\pi\co \bbR^2\to \bbR^2/ \bbZ^2$ is the quotient map.

The fibration sequence $S^1\to N^3 \to \bbR^2/\bbZ^2$ gives rise to a short-exact sequence on fundamental groups:
$$1\maps \bbZ\maps H \srm{q_*} \bbZ^2\maps 1.$$
Covering space theory gives canonical identifications of $\pi_1 (\bbR^3/H, [(0,0,0)])$ and $\pi_1 (\bbR^2/\bbZ^2, [(0,0)])$ with $H$ and $\bbZ^2$ (respectively), and under these identifications the  map $q_*$
is simply
$$\br{\begin{array}{rrrr}  1 & a  & c     \\
					  0 &  1& b     \\
					   0& 0  &1       
	 \end{array}}  \goesto (a,b).$$
The  kernel of $q_*$ is generated by 
$$Z = \br{\begin{array}{rrrr}  1 & 0  & 1     \\
					  0 &  1& 0     \\
					   0& 0  &1       
	 \end{array}},$$
which is the commutator of the elements
$$  X = \br{\begin{array}{rrrr}  1 & 1  &0     \\
					  0 &  1&0     \\
					   0& 0  &1       
	 \end{array}} \textrm{ and } Y = \br{\begin{array}{rrrr}  1 & 0  &0     \\
					  0 &  1&1     \\
					   0& 0  &1       
	 \end{array}}.$$
It is elementary to check that $X$ and $Y$ generate $H$, and it follows that $\ker(q_*)$ is precisely the commutator subgroup of $H$, giving
\begin{equation}\label{ab}H_1 (N^3; \bbZ) \isom \bbZ^2.\end{equation}
Since $Z$ commutes with both $X$ and $Y$,  we see that $\ker(q_*)$ is central, and it follows that $H$ is a nilpotent group.
  Poincar\'{e} Duality, together with  (\ref{ab}), yields:
$$H_i (N^3; \bbZ) \isom H^i (N^3; \bbZ) \isom \begin{cases} \bbZ^2 & \textrm{ if } i=1,2,\\
						 \bbZ , &  \textrm{  if }  i=0,3.  
						\end{cases} 
$$						
In particular, the cohomology of $N^3$ is torsion-free.

\subsection{Flat bundles over the Heisenberg manifold}

To understand flat bundles over $N^3$, we will use the following fact.

\begin{proposition}\label{tf} If $X$ is a finite CW complex with torsion-free integral 
cohomology, and $E\to X$ is a complex vector bundle whose Chern classes $c_i (E) \in H^{2i} (X; \bbZ)$  vanish for $i\geqs 1$, then $E$ is stably trivial.

In particular, if $M$ is a smooth  manifold with torsion-free integral 
cohomology and $E\to M$ is a vector bundle admitting a flat connection, then $E$ is stably trivial.
\end{proposition}

Note that for finite CW complexes, $H^*(X; \bbZ)$ is torsion-free if and only if $H_*(X;\bbZ)$ is torsion-free.

\begin{proof} The second statement follows from the first, since by Chern--Weil theory the Chern classes of a flat vector bundle over $M$ map to zero in $H^*(M; \bbQ)$, and when $H^*(M; \bbZ)$ is torsion-free the natural map 
$H^*(M; \bbZ)\to H^*(M; \bbQ)$
is injective.
 
To prove the first statement, consider a vector bundle
 $E\to X$   with $c_i (E) = 0$ for $i\geqs 1$.
By~\cite[Section 2.5]{AHSS} (see also~\cite[Proposition 6.10]{Hendricks}) the complex $X$--theory of $X$ is torsion-free (and finitely generated), so the natural map $\wt{K}^*(X) \to \wt{K}^*(X)\otimes \bbQ$ is injective.  Composing with the Chern character gives an injection
$$\wt{K}^*(X) \injects \wt{K}^*(X)\otimes \bbQ \srm{\isom} \wt{H}^*(X; \bbQ).$$
Since $[E]$ maps to zero under this injection, we have $[E]=0$ in $\wt{K}^*(X)$.
\end{proof}

\begin{corollary} \label{triv}
Let $G$ be  a  discrete group whose classifying space $BG$ has torsion-free integral cohomology and has the homotopy type of a CW complex $X$ of dimension at most $d$.  
Then  if  $n  \geqs d/2$, the  bundle $E_\rho$ associated to a representation $\rho\co G\to \GL (n)$ is always trivial.
\end{corollary}
\begin{proof} 
By Proposition~\ref{tf}, we know that the vector bundle associated to $E_\rho$ is stably trivial, meaning that its classifying map 
$$B\rho\co BG\heq X \to B\GL (n)$$
 becomes nullhomotopic after composing with the natural map 
$$j\co B\GL (n)\to B\GL(n+n', \bbC)$$
 for sufficiently large $n'$. Since $j$ induces an isomorphism on homotopy groups up to dimension $2n$ and a surjection in dimension $2n+1$,  the Whitehead Theorem~\cite[Section 10.3]{May-concise} shows that for all CW complexes 
$X$ of dimension at most $2n$, the map 
$$[X, B\GL (n)] \maps [X, B\GL (\infty)]$$
 is bijective.  In particular, since $j\circ B \rho$ is nullhomotopic, we conclude that $B \rho$ is itself nullhomotopic so long as $d \leqs 2n$.
\end{proof}

\begin{corollary}\label{H-rep}
Let $G = \pi_1 K$, where $K$ is an aspherical, 3--dimensional CW complex with torsion-free integral cohomology.
Then for each representation $\rho\co G\to \GL_n (\bbC)$,  the  associated vector bundle $E_\rho$ is trivial.
\end{corollary}
\begin{proof} Since $BG\heq K$, Corollary~\ref{triv} shows that $E_\rho$ is trivial whenever the degree of $\rho$ is at least 2.  Our assumptions imply that the abelianization of $K$ is free abelian, so $\Hom(G, \GL_1 (\bbC)) \isom (\GL_1 \bbC)^k$ for some $k$, and in particular is path-connected.  Hence for each $\rho\to G\to \GL_1 (\bbC)$, we have $E_\rho \isom E_1 \isom BG\cross \bbC$.\end{proof}

\subsection{Homotopy in the space of flat connections over $N^3$}\label{Heisenberg}

Let $\flatc_n (N^3)$ denote the space of flat connections on the trivial bundle $N^3\cross \U(n)$ (or, equivalently, the space of flat unitary connections on $N^3 \cross \bbC^n$).  More precisely, $\flatc_n (N^3)$ will denote the subspace of flat connections inside the  Sobolev completion, with respect to a sufficiently strong Sobolev norm, of the space of all smooth connections on $N^3\cross \U(n)$, as in~\cite[Section 5]{Baird-Ramras-smoothing}.  In this section we study the homotopy groups of $\flatc_n (N^3)$ as $n\to \infty$, using Lawson's calculation of $\K_* (H)$~\cite{Lawson-thesis, Lawson-prod, Lawson-simul}.

It was proven in~\cite[Corollary 1.3]{Baird-Ramras-smoothing} that if $M^d$ is a closed, smooth, aspherical $d$--manifold with $H^3 (M; \bbQ)\neq 0$, then  $\flatc_n (M)$ has infinitely many path components so long as $n\geqs (d+1)/2$.  In particular, $\flatc_n (N^3)$ has infinitely many path components so long as $n\geqs 2$. 
For manifolds $M$ of dimension $d> 3$,~\cite[Corollary 1.3]{Baird-Ramras-smoothing} also gives cohomological lower bounds on the rank of $\pi_m \flatc_n (M)$ for $0 < m \leqs d-3 $, but for 3-manifolds no information about the homotopy groups $\pi_m \flatc_n (M)$ ($m\geqs 1$) is obtained through the methods of that paper.  

In this section, we will show that the homotopy groups of $\flatc_n (N^3)$ are in fact very large.  Moreover, while the classes in $\pi_m \flatc_n (M)$ produced by the methods in~\cite{Baird-Ramras-smoothing} all admit representatives lying inside a single gauge orbit, we produce  classes in $\flatc_n (N^3)$ that do not admit such representatives (although see Question~\ref{go}).

This result shows a sharp contrast between flat connections over 3--manifolds  and over surfaces, where Morse theory for the Yang--Mills functional leads to the conclusion that spaces of flat connections are highly connected, with vanishing homotopy in a range tending to infinity with the rank of the underlying bundle. For details  and  precise results, see~\cite{Ramras-surface, Ramras-YM} and the references therein.

\begin{definition}\label{htpy-def}
Given a space $X$ together with a choice of representatives $\{x_C\}_{C\in \pi_0 (X)}$ for the path components of $X$, we define
$$\wt{\pi}_n (X) = \bigoplus_{C\in \pi_0 (X)} \pi_n (X, x_C).$$
Up to isomorphism, this group is independent of the chosen representatives $x_C$.
\end{definition}

Recall that the gauge group $\mcG = \Map (N^3, \U(n))$ acts on the space of all  connections on $N^3 \cross \U(n)$, and this action preserves the subspace $\flatc_n (N^3)$. More precisely, $\mcG$ is the Sobolev completion of the space of smooth maps with respect to the appropriate Sobolev norm.  The based gauge group $\mcG_0 \leqs \mcG$ is the kernel of the restriction map $\mcG \to \U(n)$ induced by evaluation at a fixed basepoint $x\in N^3$.  The holonomy map induces a fibration sequence (in fact, a principal $\mcG_0$--bundle)
\begin{equation}\label{hol} \mcG_0 \srm{i_A} \flatc_n (N^3) \xmaps{\Hol} \Hom(H, \U(n))\end{equation}
for each $n$.  The first map in this sequence is simply the inclusion of the gauge orbit of some flat connection $A\in \flatc_n (N^3)$, and sends $g\in \mcG_0$ to $g\cdot A$.  
We call maps of the form $i_A$, and their induced maps on $\wt{\pi}_*$, (\e{based}) \e{gauge orbit inclusions}.

We need the following result of Lawson.

\begin{proposition}[{\cite[Section 4.2]{Lawson-simul}}]\label{Lawson-H} 
For each $m\geqs 0$, the group $\K_m (H)$ is free abelian of countably infinite rank.
\end{proposition}

\begin{theorem}\label{unbdd}  Given $m, R\geqs 1$, there exists $n_0\geqs 1$ such that for all $n\geqs n_0$,  $\wt{\pi}_m (\flatc_n (N^3))$ contains a subgroup $F$ satisfying:
\begin{enumerate}
\item The abelianization of $F$ is free of infinite rank;
\item $\Hol_* (F) \leqs \wt{\pi}_m \Hom(H, \U(n))$ has rank  $R$; and
\item Only the trivial element of $F$ is in the image of a based gauge-orbit inclusion.
\end{enumerate}
When $m=1$, and when $m$ is even, $(3)$ can be strengthened by replacing the based gauge group by the full gauge group.
\end{theorem}
\begin{proof} Since $\wt{K}^{-m} (BH) \isom K^{-m} (N^3)$ is finitely generated, 
Proposition~\ref{Lawson-H} implies that
the kernel of 
$\wt{\alpha}_m$
is  free abelian of countably infinite rank (note here that subgroups of free abelian groups -- of any rank -- are free~\cite[Appendix 2]{Lang-algebra}).  

Now assume $m\geqs 1$.  By Theorem~\ref{Kdef}, the natural map
$$[S^m, \Rep(H)] \maps \rh_m (\Rep (H))\isom \rK_m (H)$$
is surjective.  This means that we can choose  families 
$$\rho_i \co S^m \to \Hom(H, \U(n_i)),$$ $i=1, 2, \ldots$, such that the associated classes in $\rh_m (\Rep (H))$ are linearly independent. We have  $\wt{\alpha}_*[\rho_i] = [B \circ \rho_i]$, and $B \circ \rho_i$ is nullhomotopic by Corollary~\ref{H-rep}.

Now fix an integer $R>0$.   Let $n_0 = \max(n_1, \ldots, n_R)$.  Then for $i = 1, \ldots, R$, and for each $n\geqs n_0$, the map
$$S^m \xmaps{B\circ \rho_i}  \Map_* (BH, B\U(n_i)) \maps \Map_* (BH, B\U(n_0))$$
is nullhomotopic.
Say $n\geqs n_0$, and 
let  $G$ be the subgroup of $\wt{\pi}_m \Hom(H, \U(n))$ generated by the elements 
$\{\langle \rho_i\rangle\}_{i=1}^R$ .
Then $G$
 surjects onto the  subgroup of $\rh_m (\Rep (H))$
generated by $\{[\rho_i]\}_{i=1}^R$, which is free abelian of rank $R$, 
so the abelianization of $G$ must have rank at least $R$ (of course when $m > 1$, the group $G$ is already abelian).

By~\cite[Remark 5.6]{Baird-Ramras-smoothing}, the boundary map on homotopy groups associated to the principal bundle (\ref{hol})
can be identified (up to isomorphism) with the map 
$$B_*\co \pi_* \Hom(H, \U(n))\maps \pi_* \Map_* (BH, B\U(n)),$$
so each class $\langle \rho_i\rangle$  maps to zero under the boundary map for the fibration sequence (\ref{hol}) (with appropriately chosen basepoints). Hence there exist maps $$\alpha_i \co S^m \to \flatc_n (N^3)$$ such that $\Hol_* (\langle \alpha_i \rangle)   = \langle \rho_i \rangle$.

The based gauge group $\mcG_0$  is weakly equivalent to the continuous mapping space 
$$\Map_* (N^3, \U(n))\heq \Map_* (BH, \U(n)),$$ and
$\Map_* (BH, B\U(n))$ is the classifying space of $\Map_* (N^3, \U(n))$ (this result is originally due to Gottlieb~\cite{Gottlieb}).  Hence  for $*\geqs 1$ we have
$$\pi_* \Map_* (BH, B\U(n))\isom \pi_{*-1} \mcG_0.$$

Let $A_i$ denote a flat connection in the image of $\alpha_i$, and let $\psi_i = \Hol(A_i)$. The long exact sequence associated the to the fibration sequence (\ref{hol}) ends with
$$\cdots \maps \pi_1 (\Hom(H, \U(n)), \psi_i) \srm{\partial} \pi_0 \mcG_0 \xmaps{(i_{A_i})_*}\pi_0 \flatc_n (N^3) \maps \pi_0 (\Hom(H, \U(n)).$$
Again, $\partial$ can be identified (up to isomorphism) with 
$$B_* \co \pi_1 (\Hom(H, \U(n)), \psi_i) \maps \pi_1 (\Map_* (BH, B\U(n)), B\psi_i),$$
 and the cokernel of this map has rank 1 by \cite[Corollary 1.3]{Baird-Ramras-smoothing}.  This means there exists an element $g_i\in \mcG_0$ such that the path component $[g_i]$ has infinite order in 
 $$ \pi_0 \mcG_0  \isom \pi_1 (\Map_* (BH, B\U(n)), B\psi_i)$$
and the subgroup of $\pi_0 (\mcG_0)$ generated by $[g_i]$ intersects the image of $\partial$ trivially. 
In general, given a (left) principal bundle $K$--bundle
$$K \maps P \srm{q} B,$$
if  $k_1, k_2 \in K$ and $p\in P$ are elements such that
$[k_1 \cdot p]=[k_1 \cdot  p]$  in $\pi_0 P$, then  $[k_1^{-1} k_2]$ is in the image of the boundary map
$\pi_1 (B, q(p)) \srm{\partial} \pi_0 (K)$ (where we identify $K$ with $q^{-1} (p)$ via $p\cdot k \goesto k$).
This implies that for each fixed $i\in \{1,\ldots, R\}$, the path components
$[A_i], [g_i \cdot A_i],  [g_i^2 \cdot A_i], \cdots$
are distinct in $\pi_0 \flatc_n (N^3)$.

Let $F\leqs \wt{\pi}_m \flatc_n (N^3)$  be the subgroup generated by the elements 
$\langle g_i^k\cdot \alpha_i\rangle$, $i=1, \ldots, R$, $k = 0, 1, \ldots$.  We claim that $F$ satisfies the conditions in the theorem.  First, note that $\Hol_* (F) = G$, so $F$ satisfies (2).  

Next, we show that the abelianization of $F$ is freely generated by the elements $\langle g_i^k\cdot \alpha_i\rangle$.  For simplicity, we give the argument when $m>1$, in which case $F$ is already abelian; the argument for $m=1$ just requires notational changes.
 If $\sum_{i, k} \lambda_{i,k} \langle g_i^k \cdot \alpha_i\rangle=0$, then summing the terms whose images lie in a particular path component $C$ of $\flatc_n (N^3)$ will also give zero.  By choice of $g_i$, such a sum contains at most one term of the form $\lambda_{i,k} \langle g_i^k \alpha_i\rangle$ for each $i$.  Thus we have 
$\sum_i \lambda_{i, k_i} \langle g_i^{k_i} \alpha_i\rangle = 0$ for some collection of natural numbers $k_i$ ($i=1, \ldots, R$), and every term from the original sum whose image lies in $C$ appears in this new sum.  But now
$$ 0 =  \Hol_*  (0) = \Hol_* \left(\sum_i \lambda_{i, k_i} \langle g_i^{k_i} \alpha_i\rangle\right) =\sum_i \lambda_{i, k_i} \langle\rho_i\rangle,$$ 
and   linear independence of the elements $\langle\rho_i\rangle$ implies that $\lambda_{i, k_i} = 0$ for each $i$.  Applying this argument to each path component, we see that all the coefficients $\lambda_{i,k}$ must be zero, as desired.

Finally, we consider gauge-orbit inclusions.
Say $f\in F$ is in the image of a based gauge-orbit inclusion.  We need to prove that $f=0$.  Again, for notational convenience we work in the case $m>1$.
We can write $$f = \sum_{i, k} \lambda_{i,k} \langle g_i^k \cdot \alpha_i\rangle,$$ and it suffices to consider the case in which all the maps $g_i^k \cdot \alpha_i$ land in the same  path component of $\flatc_n (N^3)$.  Then, as before, there can be at most one term in this sum for each $i$, so we can write $$f = \sum_i \lambda_{i, k_i}\langle g_i^k \cdot \alpha_i\rangle.$$  
Since $f$ is in the image of a based gauge-orbit inclusion $\mcG_0 \to \flatc_n (N^3)$, and the composite 
$$\mcG_0 \maps \flatc_n (N^3)  \maps \Hom(H, \U(n))$$ 
is constant, 
we have $\sum_i  \lambda_{i, k_i} (\langle \rho_i\rangle ) = 0$.  But this element maps to $\sum_i  \lambda_{i, k_i}  [\rho_i]$ in $\rh_m (\Rep(H))$, 
which implies that all of the coefficients $ \lambda_{i, k_i}$ are in fact zero.

Finally, we consider the orbits of the full gauge group.
Evaluation at a point gives a split fibration 
\begin{equation}\label{gf}\mcG_0 \to \mcG\to \U(n),
\end{equation}
and since $\U(n)$ is path-connected, $\wt{\pi}_* \mcG \isom \wt{\pi}_* \mcG_0 \cross \pi_* \U(n)$ for each $*\in \bbN$.  
Now consider the case where $m=2l$ is even. 
In constructing the subgroup $F$, we are free to choose $n_0$ large with respect to $2l$, and hence we may assume that that $\pi_{2l} \U(n) = 0$ whenever $n\geqs n_0$.  Thus when $m$ is even, the image of a full gauge-orbit inclusion is the same (in homotopy) as the image of the corresponding based gauge-orbit inclusion.  

To obtain the desired conclusion when $m=1$,  it suffices to show that for every gauge-orbit inclusion, the composite
$$\U(n) \srm{s} \mcG \maps \flatc_n (N^3) \maps \Hom(H, \U(n))$$
(where $s$ is a splitting of (\ref{gf}))
 induces the zero map on $\pi_1$.  The image of this composite map lies inside a single conjugation orbit inside $\Hom(H, \U(n))$, so this  composite factors through the projection $\U(n) \to \U(n)/Z$, where $Z \isom \U(1)$ denotes the center of $\U(n)$.  Since the inclusion $Z\injects \U(n)$ induces isomorphisms on $\pi_0$ and $\pi_1$, we see that $\pi_1 (\U(n)/Z) = 0$.
\end{proof}

Here are some natural questions regarding the above results.

\begin{question} $\label{go}$ Can the groups $F$ in Theorem~\ref{unbdd} be constructed more explicitly? 
When $m$ is odd and greater than 1, can they be constructed so as to avoid non-trivial elements in the image of the full gauge-orbit inclusions $\mcG\to \flatc_n (N^3)$ $?$
\end{question}

\begin{question}\label{hd} Do the results of this section  extended to higher-dimensional Heisenberg manifolds?
\end{question}

\subsection{Flat connections and the topological Atiyah--Segal map}

The arguments above suggest that when $G$ is a group such that $BG$ has the homotopy type of a closed, smooth (aspherical) manifold $M$, the homotopy fiber of the topological Atiyah--Segal map should be closely related to spaces of flat connections on bundles over $M$. Here we give one result along these lines.

First we formulate a general result about topological monoids. Let $p\co M\to N$ be a homomorphism of homotopy commutative topological monoids. We will write the operations in $M$ and $N$, and the induced pointwise multiplication operations for maps into $M$ and $N$, as $\bt$.

Let $I \subset \pi_0 N$ denote the image of $p_*\co \pi_0 M\to \pi_0 N$. Assume that $I$ is cofinal in $\pi_0 N$ 
$($that is, for each $c\in \pi_0 N$ there exists $d\in \pi_0 N$ such that $c d\in I$$)$ and that there is a map of monoids $s\co I \to N$ satisfying $[sc] = c$ for each component $c\in I$, where $[sc]\in \pi_0 N$ denotes the path component of $sc\in N$. In other words, $s$ is a partially defined section of the tautological map $N\to \pi_0 N$. Then the collection of homotopy fibers $F := \coprod_{c\in I} \hofib_{sc} (p)$ inherits a topological monoid structure from $M$ and $N$  -- here it is critical that the basepoints $\{sc\,:\, c\in I\} \subset N$ is a submonoid.

\begin{proposition}\label{hofib} 
In the setting above, 
for each $m>0$, there is a surjection
$$\rh_m F \maps \ker (\Omega Bp \co \pi_* (\Omega BM, *)\to \pi_* (\Omega BN, *)).$$
\end{proposition}

Note that if   all elements in $M$ and $N$ are strongly anchored, then the above kernel agrees with
$\ker (\Omega Bp_* \co \pi_* (\Omega BM, *)\to \pi_* (\Omega BN, *))$.

\begin{proof}  The natural maps from the homotopy fibers to $M$ induce maps of monoids $[S^m, F]\to [S^m, M]$ and $\rh_m F \to \rh_m M$, and it is immediate that the composite map $\rh_m F\to \rh_m M\to \rh_m M$ is zero. It remains only to show that if $\mu\co S^m\to M$ maps to zero in $\rh_m N$, then it lifts, up to homotopy, to $F$. Without loss of generality, we may assume that $p(\mu(1)) = sc$ for some $c\in I$. Now $p_*[\mu] = 0$ and cofinality of $I$ imply that $[\mu\bt sc'] = [sd]$ for some $c', d\in I$ (here   $sc'$ and $sd$ denote constant maps with these values). Letting $H\co S^m \cross I \to M$ be a nullhomotopy from $\mu\bt sc'$ to $sd$, the pair $(\mu\bt sc, H)$ defines a map $S^m\to \hofib_{sc} (p)$, whose image in $\rh_m M$ is $[\mu\bt sc] \equiv [\mu]$.
\end{proof}

We note that while $\pi_k \hofib_* (\Omega Bp)$ also surjects onto the above kernel, we do not expect an isomorphism $\rh_m F\isom \pi_m \hofib_* (\Omega Bp)$ in general:  one might try to construct a map $\rh_{m+1} N \to \rh_m F$ using the boundary maps for the homotopy fiber sequences $\hofib_{sc} (p)\to M\to N$, but for each component of $N$, there may be many such boundary maps, coming from  different path components, and hence choices of basepoint, in $\hofib_{sc} (p)$.

Before sketching the proof of Proposition~\ref{hofib}, we consider what it tells us about flat connections. 
Let $G$ be the fundamental group of a closed, aspherical 3-manifold $X$ with torsion-free cohomology, e.g. $G$ could be the Heisenberg group $H$.
Then setting $M = \Rep(G)$, $N = \V (BG)$, and $p = B$, Corollary~\ref{H-rep} tells us that the image of $B_*$ on components is just the components of nullhomotopic maps in $\V (BG)$, which is cofinal in $\pi_0 \V(BG)$ since every bundle over $BG$ is a direct summand of a trivial bundle. Now we may define $s$ by choosing the constant maps $BG\to BU(n)$ to represent their components. By~\cite[Theorem 5.5]{Baird-Ramras-smoothing}, the homotopy fibers of $B$ are weakly equivalent to spaces of flat connections over the manifold $X$, meaning that there are natural zig-zags of maps between these spaces that induce isomorphisms on $\pi_0$ and on homotopy groups for all compatible choices of basepoints. Since the unbased homotopy sets are naturally in bijection with the $\pi_1$--orbits of the unbased homotopy sets, we obtain a bijection $[S^m, F] \isom [S^m, \flatc (X)]$, where 
$$\flatc (X):= \coprod_n \flatc_n (X).$$
Hence $\rh_m  \flatc (X)$ surjects onto the kernel of the reduced topological Atiyah--Segal map. In the case of the Heisenberg manifold $N^3$, we saw above that this kernel is a non-finitely generated free abelian group. In summary, we have:

\begin{corollary}\label{fl-mon} The monoid $\rh_m \flatc (N^3)$ has infinite rank for each $m>0$.
\end{corollary}


\section{Multiplicativity of the topological Atiyah--Segal map}\label{image-sec}

In this section, we upgrade the topological Atiyah--Segal map to a map of $E_\infty$ ring spectra, so that the induced map $\alpha_*$ on homotopy becomes a ring homomorphism.  This additional structure allows us to deduce further constraints on the image of $\alpha_*$ for groups satisfying Kazhdan's property $(\e{T})$ and for the discrete Heisenberg group. In particular, in Section~\ref{T-subsec} we prove the following result regarding families of flat vector bundles.

\begin{theorem}\label{T-thm} Let $G$ be a discrete group satisfying Kazhdan's property $(\e{T})$, and assume that $BG$ has the homotopy type of a finite CW complex.  Then for every family
 $\rho\co S^{m}\to \Hom(G, \U(n))$,
the bundle $E_\rho$ represents a torsion class in $\wt{K}^0 (S^m\cross BG)$.
\end{theorem}

Note that by Lemma~\ref{tf}, when $H^*(BG; \bbZ)$ is torsion-free the conclusion of Theorem~\ref{T-thm} can be strengthened: $E_\rho$ is in fact stably trivial.

 There are many interesting groups to which this result applies.  All torsion-free word hyperbolic groups admit finite CW models for $BG$, built using Rips complexes~\cite[Corollary 4.12]{Short-hyperbolic}.  There are many such groups with property (T), including cocompact, torsion-free lattices in Sp$(n, 1)$~\cite{Bekka-T}.

\subsection{Bipermutative structures}\label{E-infty-sec}
Kronecker product of matrices makes the unitary permutative action sequences giving rise 
to $\K(G)$ and $\mcK (BG)$ into \e{bipermutative} action sequences, in the sense described in  Section~\ref{ring-sec}; the details are just a routine extension of the computations in May~\cite[VI \S5]{May-577}.
We thus obtain functors $\K_\otimes$ and $\mcK_\otimes$ from the category of discrete groups to the category of $E_\infty$ ring spectra, which become naturally equivalent to $\K$ and $\mcK$ after applying the forgetful functor to spectra.

\begin{theorem}\label{TAS-mult} There is a natural transformation $\alpha^\otimes$ between the functors $\K_\otimes$ and $\mcK_\otimes$, which becomes equivalent to $\alpha$ after applying the forgetful functor from $E_\infty$ ring spectra to spectra.  In particular, $\alpha_*$ is a homomorphism of unital rings, and $\wt{\alpha}_*$ is a homomorphism of non-unital rings.
\end{theorem}

\begin{proof}  The desired natural transformation is again induced by the simplicial classifying space functor $B$, which respects the multiplicative structure as well as the additive structures (by functoriality, essentially).  The statement regarding $\wt{\alpha}_*$ follows from the fact that $\wt{K}_* (G)$ and $\wt{\mcK}_* (BG)$ are simply the kernels of the compatible (ring) homomorphisms induced by the inclusion $\{1\} \to G$.   
\end{proof}

The defect in this construction is that while the homotopy groups  $\mcK_* (X)$  agree \e{additively} with the complex topological $K$--theory of $X$,  
the ring structure 
is not immediately accessible in general, and we postpone discussion of this point to future work.
Nevertheless, applications of Theorem~\ref{TAS-mult} are provided in Section~\ref{image-sec} below,  based on the following (rather limited) information regarding the rings $\mcK_* (X)$.
When $X = \{*\}$,   May~\cite[VIII \S2]{May-577} showed that the ring $\mcK_* (*)$ is isomorphic to $\pi_*  \ku  = \bbZ [\beta]$, where $\beta\in \pi_2 (\ku)\isom \bbZ$ is a generator.  In other words, this is the standard Bott-periodic connective $K$--theory ring of a point.  For each finite CW complex $X$, the projection $X\to  *$ induces an injective ring map $ \mcK_* (*)  \to  \mcK_* (X)$, which embeds the ring $\pi_* \ku$ in $\mcK_* (X)$.
 
%
%
%
%
%
\subsection{Deformation $K$--theory and spaces of irreducible representations}\label{SS-sec}
We need to review some of Lawson's results from~\cite{Lawson-prod, Lawson-simul}, which allow one to compute (unitary) deformation $K$--theory from homological information about spaces of irreducible representations. 

First, consider the space
$$\ol{\Rep} (G)  = \coprod_n \Hom(G, \U(n))/\U(n),$$
where the quotient on the right is taken with respect to the conjugation action.  Block sum makes this into a strictly commutative topological monoid.  In fact, the sequence of spaces $$\Rep_n (G) = \Hom(G, \U(n))/\U(n)$$ 
form a \e{bipermutative} action sequence for the trivial groups $G_n = \{1\}$, using Kronecker product to define the multiplicative structure. The associated  $E_\infty$ ring spectrum $\Rdef (G)$ satisfies $\Omega^\infty \Rdef (G) \heq \ol{\Rep} (G)$ by Proposition~\ref{nerve}.  
We define
$$\Rdef_* (G) = \pi_* \Rdef (G).$$
The quotient maps $\Hom(G, \U(n)) \to  \Hom(G, \U(n))/\U(n)$ respect block sum and Kronecker product, so we obtain an induced map of $E_\infty$ ring spectra
$$\K(G) \maps \Rdef (G).$$
At this point, we need to pass from the category of  $E_\infty$ ring spectra
to the category of $\bS$--algebras, as constructed in~\cite{EKMM}.  The desired functor is discussed in~\cite[II.3]{EKMM}, and for us the important point is that it induces an isomorphism on the underlying homotopy rings.  We will continue to use the same notation for our $E_\infty$ ring spectra and their associated $\bS$--algebras, but it should be noted that smash products will be formed in the derived category of $\bS$ modules or $\ku$--modules, as appropriate.

In~\cite{Lawson-prod}, it is shown that  when $G$ is finitely generated, there is an equivalence $H\bbZ \sm \K(G) \heq \Rdef (G)$.  Fix a generator of $\pi_2 \ku$ and a map $\beta\co \bS^2 \to \ku$ representing it.  We call $\beta$ the Bott element. The natural map $G\to \{1\}$ induces a map $\ku\to \K(G)$, and the image of $\beta$ (which we still denote simply by $\beta$) is the Bott element in $\K_2 (G)$. Smashing 
$\bS^2 \srt{\beta} \ku$ with $\ku$ induces a map  
$$\Susp^2 \ku \maps \ku$$
which we call the Bott map (and also denote by $\beta$).  Bott periodicity implies that the homotopy cofiber of $\beta$ is the Eilenberg--MacLane spectrum $H\bbZ$.

Smashing the homotopy cofiber sequence 
$$\Susp^2 \ku\maps \ku \maps H\bbZ$$ 
with $\K(G)$ (as $\ku$--modules; that is, applying $\sm_\ku$) and taking homotopy groups now gives a long exact sequence of the form
\begin{eqnarray}\label{Lawson-LES}
\hspace{.3in} \cdots\srm{\partial} \ \K_* (G) \srm{\beta} \K_{*+2} (G) \maps \Rdef_{*+2} (G)  \srm{\partial} \K_{*-3} (G)  \srm{\beta} \cdots,
\end{eqnarray}
wherein the map $\beta$ is multiplication by the Bott element

Lawson also developed a spectral sequence for computing $\Rdef_* (G)$ from the integral homology of the spaces
$$\Irr_n (G) \defn \Rep_n (G) /  \Sum_n (G),$$
where $\Sum_n (G)$ denotes the subspace of reducible representations.  Note that $\Irr_n (G)$ is the one-point compactification of complement of  $\Sum_n (G)$ in $\Rep_n (G)$, and this complement is precisely the subspace of irreducible representations.  The spectral sequence is constructed by considering the tower of spectra 
$$\displaystyle{* = \Rdef_{\leqs 0} (G) \maps \Rdef_{\leqs 1} (G) \maps \Rdef_{\leqs 2} (G) \maps \  \cdots,}$$
where $\Rdef_{\leqs k}$ is the spectrum associated to the subspaces of $\Rep_n (G)$ consisting of representations whose irreducible summands all have dimension at most $k$; note that these subspaces provide a submonoid of $\ol{\Rep(G)}$, and in fact a permutative subsequence of $(\Rep_n (G))_{n=0}^\infty$.  The homotopy colimit of this sequence is $\Rdef (G)$, and Lawson proves that there are homotopy cofiber sequences of spectra
$$\Rdef_{k-1} (G) \maps \Rdef_{k} (G)\maps H\bbZ \sm \Irr_{k}  (G)$$
for each $k\geqs 1$.  

In general, a sequence of spectra 
$$X_0 \srm{f_0} X_1 \srm{f_1} X_2\srm{f_2}  \cdots$$
 gives rise to an
exact couple
$$\xymatrix{\displaystyle { \bigoplus_{q,p} \pi_q X_p} \ar[rr]^{\displaystyle{\oplus (f_p)_*}} &&   \displaystyle{\bigoplus_{q,p} \pi_q X_p} \ar[dl]  \\ & 
\displaystyle{\bigoplus_{q,p}} \pi_q (\hocofib f_p) \ar[ul]^\partial}$$
and hence to a spectral sequence of the form
\begin{equation*}E^1_{p,q} = \pi_{p+q} (\hocofib f_p ) \implies \pi_{p+q} \hocolim_i X_i,
\end{equation*}
with differentials 
 $$d^r_{p,q} \co E^r_{p, q} \maps E^r_{p-r, q+r-1}.$$
Since $H\bbZ$ is the spectrum representing integral homology, in the case at hand we obtain a spectral sequence
\begin{equation}\label{SS}E^1_{p,q} = \wt{H}_{p+q} ( \Irr_{p} (G); \bbZ ) \implies \pi_{p+q} \Rdef (G).\end{equation}
Note that with this indexing, the spectral sequence can be non-zero in the quadrant where $p, q\geqs 0$ and in the region where $-p\leqs q < 0$.
%
\subsection{Groups satisfying Kazhdan's property (T)}\label{T-subsec}

Property (T) has been widely studied since its introduction by Kazhdan in the late 1960s.   Loosely speaking, property (T) is a weak rigidity property for (possibly infinite-dimensional) unitary representations of locally compact groups.   See~\cite{Bekka-T} for background.

We need a lemma regarding unitary representations of property (T) groups.\footnote{I learned this result from Rufus Willett.}  

\begin{lemma} \label{T-comps}
Let $G$ be a discrete group with property $(\e{T})$.   Then the space $\Hom(G, \U(n))/\U(n)$ is a finite, discrete space.
\end{lemma}
\begin{proof} By Wang~\cite[Theorem 2.5]{Wang}, if $\rho\co G\to \U(n)$ is irreducible, then the path component of $\rho$ in $\Hom(G, \U(n))$ coincides with its conjugation orbit $O_\rho$. 
Since $\Hom (G, \U(n))$ is compact and triangulable, it has finitely many path components, so we conclude that there are only finitely many irreducible unitary representations in each dimension.  Since every unitary representation is a direct sum of irreducibles, finiteness of $\Hom(G, \U(n))/\U(n)$ follows immediately, and discreteness follows as well since this space is Hausdorff.
\end{proof}

\begin{proposition}\label{K-T}
Let $G$ be a discrete group satisfying property $(\e{T})$.  Then $\K_{2m+1} (G)$ is trivial for all $m\geqs 0$, and the iterated Bott map
$$\beta^m_* \co \K_0 (G) \isom \Gr (\pi_0 \Rep (G)) \maps \K_{2m} (G)$$
is an isomorphism for all $m\geqs 1$.
\end{proposition}
\begin{proof}  
Lemma~\ref{T-comps} implies that $\Hom(G, \U(n))/\U(n)$ is \e{discrete} for all $n$, so its cohomology vanishes in positive dimensions. Discrete groups with property (T) are finitely generated~\cite[Theorem 1.3.1]{Bekka-T}, so the spectral sequence (\ref{SS}) yields $\pi_* (\Rdef (G)) = 0$ for $*>0$. The long exact sequence (\ref{Lawson-LES}) shows that  $\K_1 (G) \isom \Rdef_1 (G) = 0$ and that $\beta_* \co \K_m (G) \to K_{m+2} (G)$ is an isomorphism for $m\geqs 0$.
\end{proof}  

We need a standard fact about flat bundles, coming from Chern--Weil theory.

\begin{lemma}\label{flat} Let $G$ be a discrete group such that $BG$ has the homotopy type of a finite CW complex, and let $\epsilon^n$ denote the trivial bundle $BG \cross \U(n)$.  Then for every representation $\rho\co G\to \U(n)$, the class $[E_\rho] - [\epsilon^n]$ is torsion in $\wt{K}^0 (BG)$.
\end{lemma}

For a complete proof (of a much more general statement) 
see~\cite[Theorem 3.5]{Baird-Ramras-smoothing}.

\begin{proposition}$\label{T-prop}$
Let $G$ be a discrete group satisfying property $(\e{T})$, and assume that $BG$ has the homotopy type of a finite CW complex.  Then the reduced unitary topological Atiyah--Segal map 
$$\wt{\alpha}_m \co \rK_* (G) \maps \wt{K}^{-m} (G)$$
is zero  when $m$ is odd, and its image is torsion  when $m$ is even.
\end{proposition}

\begin{proof} For $m$ odd, this is immediate from Proposition~\ref{K-T}, so we consider the even case.  
By Theorem~\ref{TAS}, the image of $\alpha_0$ 
consists of $K$--theory classes of  the form $[E_\rho]$, where $\rho\co G\to \U(n)$ is a single representation.  
Lemma~\ref{flat} implies that  the image of $\alpha_0$ becomes torsion after modding out the summand $\pi_0 \ku \isom \bbZ$ corresponding to the trivial bundles. But since  $\alpha_0 = \wt{\alpha}_0 \oplus \textrm{Id}_{\pi_0 \ku}$ (see (\ref{alpha-split})), this quotient is isomorphic to the image of $\wt{\alpha}_0$.

Now consider the  commutative diagram 
\begin{equation}\label{sq}
\xymatrix{ \K_{2m} (G) \ar[r]^-{\alpha_{2m}} & \mcK_{2m} (BG) \\
\K_0 (G) \ar[u]_-{\isom}^-{\cdot \beta^m} \ar[r]^-{\alpha_0} & \mcK_{0} (BG), \ar[u]_-{\cdot \alpha_{2m} (\beta^m)}
}
\end{equation}
where the vertical maps are given by multiplication. 
Each group in the diagram contains a $\bbZ$ summand arising from the homotopy of $\ku$, via the maps induced by $G\to \{1\}$ and $BG\to \{*\}$, and these summands are complementary to the reduced subgroups.
All four maps in the diagram are isomorphisms when restricted to these $\bbZ$ summands, 
so Proposition~\ref{T-prop} implies that the image of the lower composite, and hence also the upper composite, has rank 1.  Since $\cdot \beta^m \co \K_0 (G) \to \K_{2m} (G)$ is an isomorphism, we conclude that the image of $\alpha_{2m}$ has rank 1, and finally that 
 the image of $\wt{\alpha}_{2m}$ is torsion, as desired.
\end{proof}

 \begin{proof}[Proof of Theorem~\ref{T-thm}] By Proposition~\ref{T-prop}  and Theorem~\ref{TAS}, we know that
 $$\wt{\alpha}_m ([\rho]) = \pi_*^{-1} \left([E_\rho] - [E_{\wt{\rho(1)}}]\right)$$ 
 is torsion in $\wt{K}^0 (S^{m} \sm BG)$.  But the projection $\pi \co S^{m} \cross BG \to S^{m}\sm BG$ induces a (split) injection on reduced $K$--theory, so for some $k\geqs 1$ we have $k[E_\rho]  =  k[E_{\wt{\rho(1)}}]$ in $K^0 (S^{m} \cross BG)$. Hence it will suffice to show that $l[E_{\wt{\rho(1)}}]$ represents a torsion class in reduced $K$--theory for some $l\geqs 1$.  This holds for the bundle $E_{\rho(1)} \to BG$ by   Lemma~\ref{flat}, and since $E_{\wt{\rho(1)}}$ is a pullback of $E_{\rho(1)}$, the proof is complete.
  \end{proof}

 \begin{remark}
Theorem~\ref{T-thm} provides a partial answer to~\cite[Question 3.20]{RWY}.  If $G$ satisfies the hypotheses of Theorem~\ref{T-thm}, then no non-trivial class in the rational $K$--homology of $BG$ can be detected, in the sense of~\cite[Definition 3.4]{RWY}, by a spherical family of representations.  However, it remains possible that non-trivial classes can be detected by non-spherical families of representations.
 \end{remark}  
%
%
%
  
\subsection{The topological Atiyah--Segal map for the Heisenberg group}\label{H-subsec}
In the case of the 3--dimensional integral Heisenberg group $H$, Lawson showed in~\cite{Lawson-thesis, Lawson-prod} that $\Rdef_m (H) = 0$ for $m\geqs 3$.  
In dimension 1, Theorem~\ref{BR} tells us that the image of $\alpha_1 = \alpha_1^H$ has rank at most $ \beta_1 (N^3) = 2$.
Reasoning similar to the proof  of Proposition~\ref{T-prop}  yields the following result.

\begin{proposition}$\label{alpha-H}$
The image of the  unitary topological Atiyah--Segal map
$$\alpha_{2m+1} \co \K_{2m+1}  (H) \maps K^{-(2m+1)} (BH) \isom  K^{-(2m+1)} (N^3)$$
has rank at most $2$ for each $m\geqs 0$.  Hence $\alpha_*$ is not surjective in odd dimensions. 
\end{proposition}

%
%

\def\cprime{$'$}
\providecommand{\bysame}{\leavevmode\hbox to3em{\hrulefill}\thinspace}
\providecommand{\MR}{\relax\ifhmode\unskip\space\fi MR }
\providecommand{\MRhref}[2]{%
  \href{http://www.ams.org/mathscinet-getitem?mr=#1}{#2}
}
\providecommand{\href}[2]{#2}

\end{document}